\documentclass[12pt]{amsart}
\usepackage[margin=1.1in]{geometry}

\usepackage{amsmath}
\usepackage{amsthm}
\usepackage{amstext}
\usepackage{amssymb}
\usepackage{amsfonts}
\usepackage{young}
\usepackage{multicol}
\usepackage{mathrsfs}
\usepackage{stmaryrd}
\usepackage{color}
\usepackage{tabmac}
\usepackage{mathrsfs}
\usepackage{tikz}
\usetikzlibrary{matrix,arrows,decorations.pathmorphing} 
\usepackage{ytableau}

\newtheorem{theorem}{Theorem}[section]
\newtheorem*{theorem*}{Theorem}
\newtheorem{lemma}[theorem]{Lemma}

\newtheorem{proposition}[theorem]{Proposition}
\newtheorem{corollary}[theorem]{Corollary}
\newtheorem{definition}[theorem]{Definition}
\newtheorem{conjecture}[theorem]{Conjecture}
\newtheorem{problem}[theorem]{Problem}

\theoremstyle{definition}
\newtheorem{example}[theorem]{Example}
\newtheorem{remark}[theorem]{Remark}

\def\y{{\tt y}}
\def\x{{\tt x}}
\def\R{{\tt R}}
\def\trop{\operatorname{trop}}
\def\Z{\mathbb{Z}}
\def\Q{\mathbb{Q}}
\def\C{\mathcal{Q}}
\def\LSym{\mathrm{LSym}}
\def\sl{\mathfrak{sl}}
\def\wt{\mathrm{wt}}
\def\Frac{{\rm Frac}}

\newcommand\defn[1]{{\bf #1}}

\usepackage[enableskew]{youngtab}
\def\qed{\hfill $\vrule height 2.5mm  width 2.5mm depth 0mm $}
\newcommand{\V}{{\large\textvisiblespace}}
\newcommand{\X}{{}}
\def\D{\mathscr{D}}

\author{Thomas Lam}
\address{Department of Mathematics, University of Michigan,
2074 East Hall, 530 Church Street, Ann Arbor, MI 48109-1043, USA}
\email{tfylam@umich.edu}\thanks{T.L. was supported by NSF grants DMS-1160726,. DMS-1464693, and a Simons Fellowship.}

\author{Pavlo Pylyavskyy}
\address{Department of Mathematics, University of Minnesota,
127 Vincent Hall, 206 Church Street, Minneapolis, MN 55455, USA}
\email{ppylyavs@umn.edu}\thanks{P.P. was supported by NSF grants DMS-1068169, DMS-1148634, DMS-1351590 and a Sloan Fellowship.}

\author{Reiho Sakamoto}
\address{Department of Physics, Tokyo University of Science, Kagurazaka, Shinjuku, Tokyo, 162-8601, Japan}
\email{reiho@rs.tus.ac.jp}

\title{Rigged configurations and cylindric loop Schur functions}
\begin{document}
\begin{abstract}
Rigged configurations are known to provide action-angle variables for remarkable discrete dynamical systems known as box-ball systems.
We conjecture an explicit piecewise-linear formula to obtain the shapes of a rigged configuration from a tensor product of one-row crystals.  We introduce cylindric loop Schur functions and show that they are invariants of the geometric $R$-matrix.  Our piecewise-linear formula is obtained as the tropicalization of ratios of cylindric loop Schur functions.  We prove our conjecture for the first shape of a rigged configuration, thus giving a piecewise-linear formula for the lengths of the solitons of a box-ball system. 
\end{abstract}
\maketitle

\tableofcontents

\section{Introduction}
Kashiwara \cite{Kas1} introduced crystal bases as combinatorial analogues for irreducible representations of quantum groups.  In this paper we will be concerned with tensor products of the Kirillov--Reshetikhin crystals, which we loosely call affine crystals, associated to a certain class of finite-dimensional representations of the quantum affine algebra $U_q(\hat \sl_n)$.

As a set, an affine crystal $B=B^{r_1,s_1}\otimes B^{r_2,s_2}\otimes\cdots\otimes B^{r_m,s_m}$ of type $\hat \sl_n$ is the direct product of sets of semistandard Young tableaux (with entries in $\{1,2,\ldots,n\}$), of rectangular shape $r_i \times s_i$.  Affine crystals are equipped with a remarkable combinatorial $R$-matrix $R: B^{r,s} \otimes B^{r',s'} \to B^{r',s'} \otimes B^{r,s}$ that generates an action of the symmetric group $S_m$ on $B$.

Rigged configurations are collections $$(\nu,J)=\left\{
(\nu^{(0)}),(\nu^{(1)},J^{(1)}), (\nu^{(2)},J^{(2)}),\ldots,(\nu^{(n-1)},J^{(n-1)})\right\}$$ of partitions $\nu^{(a)}$ and integer sequences $J^{(a)}$ subject to certain conditions (see Section \ref{sec:RC}). The rigged configurations were introduced in the papers \cite{KKR,KR}, motiviated by the Bethe ansatz for the isotropic Heisenberg model (see \cite{KS} for the precise relation between them).  The rigged configuration bijection $\Phi: B \to \{(\nu,J)\}$ \cite{KSS} establishes a bijection between the affine crystal, modulo the combinatorial $R$-matrix, and a set of rigged configurations.  This bijection is reviewed in Section \ref{sec:Phi}.\footnote{A Mathematica implementation of the bijection $\Phi$ is available at \cite{S:web}.}

In recent years, it has been realized that rigged configurations can be thought of as action-angle coordinates of certain discrete dynamical systems known as box-ball systems \cite{KOSTY} (see \cite{S:review} for related references).  Box-ball systems are surprisingly rich integrable systems defined using elementary rules (see Section \ref{sec:bbs}).  In the language of dynamical systems, finding the rigged configuration is akin to solving an initial value problem, and the inverse $\Phi^{-1}$ of the rigged configuration bijection has been written down explicitly as piecewise-linear functions in the case $B = B^{1,s_1}\otimes B^{1,s_2}\otimes\cdots\otimes B^{1,s_m}$ \cite{KSY,S_2006} (see \cite[Section 5]{S:review} for an introductory account). This inverse $\Phi^{-1}$ is expressed in terms of certain $\tau$-functions which play a similar role in this theory as the classical $\tau$-functions play in the traditional theory of integrable systems.

The present work is concerned with a piecewise-linear description of the map $\Phi$, in the special case that $B = B^{1,s_1}\otimes B^{1,s_2}\otimes\cdots\otimes B^{1,s_m}$ is a product of one-row crystals.  Let $p = b_1 \otimes \cdots \otimes b_m \in B$.  We let $x_j^{(i+j-1)}$ stand for the number of occurrences of the letter $i$ in $b_{m+1-j}$.  The rigged configuration $\Phi(p)$ is invariant under the combinatorial $R$-matrix action on $B$, and thus we expect each part $\nu_k^{(a)}$ of each partition (and also every rigging) in $\Phi(p)$ to be expressed as a piecewise-linear expression in $x_j^{(i)}$ invariant under the combinatorial $R$-matrix.  Since the combinatorial $R$-matrix generates an action of the symmetric group $S_m$, the answer we seek maybe thought of as a tropical analogue of a symmetric function.

There is a rational-lifting of the entire story, where combinatorial crystals are replaced by geometric crystals \cite{BeKa}, and combinatorial $R$-matrices by geometric $R$-matrices. For example, Yamada \cite{Y} considered a geometric analogue of the piecewise-linear formula \cite{HHIKTT} of the combinatorial $R$-matrix for the tensor product of one row crystals and introduced the elementary loop symmetric functions as polynomial invariants for the geometric $R$-matrix. Lam and Pylyavskyy \cite{LPloop} have studied the ring of polynomial invariants of the geometric $R$-matrix, called the ring of loop symmetric functions $\LSym_m$.  In \cite{LP} it was shown that the energy function of the affine crystal is equal to the tropicalization of a loop Schur function, the analogue of a Schur function in the ring $\LSym_m$.  The energy function plays a crucial role in affine crystal theory, especially in the relation with solvable lattice models.  This project began with the hope that formulae for the rigged configuration $\Phi(p)$ could be obtained as tropicalizations of other distinguished elements in $\LSym_m$.

In Theorem \ref{thm:cylindric} below, we introduce cylindric loop Schur functions and show that they are invariants of the geometric $R$-matrix.  These are analogues in the ring $\LSym_m$ of the cylindric Schur functions studied by McNamara \cite{M} and Postnikov \cite{P}, in a very different context.  Our main conjecture (Conjecture \ref{conj:shapes}) states that each $\nu_k^{(a)}$, as a function of $x_j^{(i)}$, has an expression of the form
$$
\nu^{(a)}_k = \trop \left(\frac{s_{\D_a(\lambda(a,k-1))}^{(0)}}{s_{\D_a(\lambda(a,k))}^{(0)}} \right)(x_j^{(i)})
$$
where $s_{\D_a(\lambda(a,k))}^{(0)}$ denotes a particular cylindric loop Schur function that we define.  For the precise definitions, see Section \ref{sec:main}.
We remark that direct computations of explicit piecewise-linear formulae
for $\Phi(p)$ for certain examples are presented in \cite{BK} (see Remark \ref{rem:BerensteinKirillov}).

We prove our main conjecture in the case that $a= 1$, that is, we are concerned with the first part of the rigged configuration $\Phi(p)$.  This case is the most important from the point of view of box-ball systems: the sequence $\nu_1^{(1)}, \nu_2^{(1)}, \nu_3^{(1)}, \ldots$ is exactly the sequence of lengths of the solitons in the corresponding box-ball system, arranged in increasing order. As a corollary, we obtain an upper bound for the number of solitons contained in a path (Corollary \ref{cor:number_of_solitons}). Another interesting property of the $a=1$ case is that we can write down the polynomials involved in this case in a very explicit form (Proposition \ref{prop:simple_formula}). Our proof relies heavily on the results of Lam and Pylyavskyy \cite{LP} and Sakamoto \cite{S}.

We now discuss possible implications of our Conjecture \ref{conj:shapes}.
The basic structure of our proof for the $\nu^{(1)}$ case (Theorem \ref{thm:main}) is as follows.
The results of \cite{S} (see Theorem \ref{th:E=Q}) expresses the partition $\nu^{(1)}$ in terms of combinatorial $R$-matrices and energy functions by using time evolutions of the box-ball systems (see the diagram (\ref{eq:def_bbs})
and Definition \ref{def:energy}).
We make use of the explicit formula for the energy \cite{LP} (see Theorem \ref{thm:LP})
to obtain Theorem \ref{thm:main}.
Now the results of \cite{S} implies that the information of $\nu^{(a)}$ for $a>1$
is essentially connected with more general crystals $B^{r,s}$ even if we only consider
the case $p\in\bigotimes_iB^{1,s_i}$.
Also the formulation of Conjecture \ref{conj:shapes} is uniform for arbitrary $\nu^{(a)}$.
On the other hand, a construction of geometric crystals for the general case
$\bigotimes_iB^{r_i,s_i}$ is still an open problem.
Therefore it is tempting to say that our Conjecture \ref{conj:shapes}
provides supporting evidence that geometric crystals for arbitrary Kirillov--Reshetikhin crystals
$\bigotimes_iB^{r_i,s_i}$ should have a beautiful structure.
We hope that our result provides a motivation for their construction.

Note that in \cite{Ok} Okado has pursued a related idea of defining rigged configurations over $\mathbb Q$, with the goal of eventually defining them over $\mathbb R$. Our formulas accomplish exactly that for the special case of tensor products of one-row crystals.

This paper is organized as follows.
In Section \ref{sec:bbs}, we review necessary background on the box-ball systems and crystal bases.
In Section \ref{sec:rig}, we review the theory of rigged configurations and the rigged configuration bijection.
In Section \ref{sec:loop}, we review the theory of loop symmetric functions.
Then we introduce the cylindric loop Schur functions and show that they belong to the ring of loop symmetric functions.
In Section \ref{sec:main}, we formulate our conjecture about the shapes $\nu^{(a)}$ of the rigged configurations.
In Section \ref{sec:first-shape}, we provide a proof for the case $\nu^{(1)}$ and discuss related subjects.

A Mathematica implementation of the main constructions in the present paper is available at \cite{S:web}.

\medskip

\noindent {\bf Acknowledgments.}  We thank Gabe Frieden for pointing out a number of errors in an earlier version of this article, including an important correction to Conjecture \ref{conj:shapes}.

\section{Box-ball systems}\label{sec:bbs}
\subsection{Definition}
The content of the present work is deeply interrelated with a remarkable
discrete soliton system called the box-ball system.
In the simplest case, the box-ball system is a dynamical
system defined over sequences of positive integers which we call paths.
The following is a typical example of the dynamics:
\begin{center}
$t=0$: \texttt{\V 332\V \V \V 42\V \V \V 4\V \V \V \V \V \V \V \V \V \V \V \V \V \V \V \V \V }\\
$t=1$: \texttt{\V \V \V \V 332\V \V 42\V \V 4\V \V \V \V \V \V \V \V \V \V \V \V \V \V \V \V }\\
$t=2$: \texttt{\V \V \V \V \V \V \V 332\V 42\V 4\V \V \V \V \V \V \V \V \V \V \V \V \V \V \V }\\
$t=3$: \texttt{\V \V \V \V \V \V \V \V \V \V 33\V 4242\V \V \V \V \V \V \V \V \V \V \V \V \V }\\
$t=4$: \texttt{\V \V \V \V \V \V \V \V \V \V \V \V 33\V 2\V 442\V \V \V \V \V \V \V \V \V \V }\\
$t=5$: \texttt{\V \V \V \V \V \V \V \V \V \V \V \V \V \V 3\V 32\V \V 442\V \V \V \V \V \V \V }\\
$t=6$: \texttt{\V \V \V \V \V \V \V \V \V \V \V \V \V \V \V 3\V \V 32\V \V \V 442\V \V \V \V }\\
$t=7$: \texttt{\V \V \V \V \V \V \V \V \V \V \V \V \V \V \V \V 3\V \V \V 32\V \V \V \V 442\V }
\end{center}
Here ``\texttt{\V}" is a substitute for the letter 1 which we regard as an empty place.
The time evolution proceeds from top to bottom as indicated by the time variable $t$.
When $t=0$, we see three distinctive chunks of lengths 3, 2 and 1.
We regard them as three solitary waves (solitons).
For $t=0,1,2$, these three solitons propagate at the velocity equal to each length.
At $t=3,4$, there are collisions of the solitons, though we regain
lengths 1, 2 and 3 solitons at $t=5,6,7$.
This is a characteristic feature of ordinary soliton systems.

Let us denote the path at time $t$ by $(T^{1,\infty})^t(p)$ where
$p$ is the initial state ($t=0$).
In order to prevent the balls moving out of the path,
we assume that the right part of $p$ is a sufficiently long sequence of the letter 1.
Then we have the following interpretation of the time evolution $T^{1,\infty}$
in terms of boxes and balls \cite{Tak}.
In this picture, we regard paths as arrays of capacity one boxes
and the letters $a>1$ as balls that can fill the boxes.
Then we introduce the operators $K_a$ $(a>1)$ which act on the paths by the following procedure:
\begin{enumerate}
\item
Move the leftmost ball $a$ to the nearest empty place on the right.
\item
Among the untouched balls $a$, move the leftmost $a$ to the nearest empty place on the right.
\item
Repeat the procedure until all the balls $a$ are moved exactly once.
\end{enumerate}
Suppose that the maximal number used in the path $p$ is $n$.
Then we define
\[
T^{1,\infty}(p)=K_2K_3\cdots K_{n}(p).
\]
When $n=2$, the system reduces to the original Takahashi--Satsuma system \cite{TS}.

\subsection{Crystals}\label{sec:R}
In order to fully clarify the mathematical structures behind box-ball systems,
Kashiwara's crystal base theory \cite{Kas1} provides a powerful tool.
Let $B^{r,s}$ be the Kirillov--Reshetikhin crystal of type $\hat{\mathfrak{sl}}_n$.
Here $r\in I_0$ where
\[
I_0:=\{1,2,\ldots,n-1\}
\]
and $s\in\mathbb{Z}_{>0}$.
As a set, $B^{r,s}$ is comprised of all rectangular semistandard tableaux of height $r$ and
width $s$, over the letters $1,2,\ldots,n$.
On $B^{r,s}$, we introduce an algebraic structure by the Kashiwara operators
$\tilde{e}_i, \tilde{f}_i$ ($i=0,1,\ldots,n$).
For the explicit forms of their actions, we refer to \cite{KN} for the $i\in I_0$ case
and to \cite{Shimo} for the $i=0$ case.  See \cite{O} for a concise introduction to the subject.

One of the nice properties of crystal bases is that they have a nice tensor product.
As a set, $B\otimes B'$ is just the product of $B$ and $B'$ and one can define the crystal structure on them in the following way.
For $b\in B^{r,s}$, define the functions
$\varepsilon_i,\varphi_i:B\longmapsto\mathbb{Z}$ by
\begin{align*}
\varepsilon_i(b)=\max\{m\geq 0\,|\,\tilde{e}_i^mb\neq 0\},\qquad
\varphi_i(b)=\max\{m\geq 0\,|\,\tilde{f}_i^mb\neq 0\}.
\end{align*}
Then the Kashiwara operators act on $B\otimes B'$
according to the following rule:
\begin{align*}
\tilde{e}_i(b\otimes b')=&
\left\{\!
\begin{array}{ll}
\tilde{e}_ib\otimes b'&\text{if }\varphi_i(b)\geq\varepsilon_i(b')\\
b\otimes\tilde{e}_ib'&\text{if }\varphi_i(b)<\varepsilon_i(b')
\end{array}
\right.,\\
\tilde{f}_i(b\otimes b')=&
\left\{\!
\begin{array}{ll}
\tilde{f}_ib\otimes b'&\text{if }\varphi_i(b)>\varepsilon_i(b')\\
b\otimes\tilde{f}_ib'&\text{if }\varphi_i(b)\leq\varepsilon_i(b')
\end{array}
\right..
\end{align*}
Note that we use Kashiwara's original convention \cite{Kas1}
for the tensor product.

The tensor product $B\otimes B'$ has two additional natural structures
called the combinatorial $R$-matrix and the energy function \cite{Kas2}.
The combinatorial $R$-matrix is the unique crystal isomorphism
which swaps the two components $R:B\otimes B'\longmapsto B'\otimes B$.
On the other hand, the energy function $H:B\otimes B'\longmapsto\mathbb{Z}$
is defined by the following axiom.
For an element $b\otimes b'\in B\otimes B'$, suppose that we have
$R:b\otimes b'\longmapsto\tilde{b}'\otimes\tilde{b}$.
Under this situation, we have the following four possibilities for
the actions of Kashiwara operators:
\begin{align*}
\begin{array}{ll}
\text{(LL)}&R(\tilde{e}_ib\otimes b')=\tilde{e}_i\tilde{b}'\otimes\tilde{b},\\
\text{(LR)}&R(\tilde{e}_ib\otimes b')=\tilde{b}'\otimes\tilde{e}_i\tilde{b},\\
\text{(RL)}&R(b \otimes\tilde{e}_ib')=\tilde{e}_i\tilde{b}'\otimes\tilde{b},\\
\text{(RR)}&R(b\otimes \tilde{e}_ib')=\tilde{b}'\otimes\tilde{e}_i\tilde{b}.
\end{array}
\end{align*}
Then the function $H$ is uniquely defined by the following relations
\begin{align*}
H(\tilde{e}_i(b\otimes b'))=
\left\{
\begin{array}{ll}
H(b\otimes b')+1&\text{if }i=0\text{ and case (LL) occurs,}\\
H(b\otimes b')-1&\text{if }i=0\text{ and case (RR) occurs,}\\
H(b\otimes b')&\text{otherwise,}
\end{array}
\right.
\end{align*}
up to some additive constant.
As a side remark, the origin of this axiom is the Yang--Baxter relation for the so-called
affine combinatorial $R$-matrices.

Following \cite{Shimo} (see also \cite{SW}) we give an explicit combinatorial description of combinatorial $R$-matrices and energy functions.
Let $Y$ be a semistandard tableau and $x$ be some positive integer.
We denote by $(Y\leftarrow x)$ the Schensted row insertion.
We define the insertion of more general words by
$(Y\leftarrow xy)=((Y\leftarrow x)\leftarrow y)$ and so on.
Denote the rows of $Y$ by $y_1,y_2,\ldots,y_r$ from top to bottom.
Then the row word $\operatorname{row}(Y)$ is defined by concatenating
rows: $\operatorname{row}(Y)=y_ry_{r-1}\cdots y_1$.
Finally, let $\xi=(\xi_1,\xi_2,\ldots)$ and $\eta=(\eta_1,\eta_2,\ldots)$
be two partitions. Then the concatenation of $\xi$ and $\eta$
is the partition $(\xi_1+\eta_1,\xi_2+\eta_2,\ldots)$.

\begin{theorem}\label{thm:Hdefn}
Let $b\otimes b'\in B^{r,s}\otimes B^{r',s'}$.
Then the image of the combinatorial $R$-matrix
\[
R:b\otimes b'\longmapsto\tilde{b}'\otimes\tilde{b}\in B^{r',s'}\otimes B^{r,s}
\]
is uniquely characterized by
\[
(b'\longleftarrow\operatorname{row}(b))=
(\tilde{b}\longleftarrow\operatorname{row}(\tilde{b}')).
\]
Moreover, the energy function $H(b\otimes b')$ is given by the number of
nodes of $(b'\longleftarrow\operatorname{row}(b))$ outside the concatenation
of the partitions $(s^r)$ and $(s'^{r'})$.
\end{theorem}
\begin{example}
Let $b\otimes b'\in B^{2,2}\otimes B^{3,3}$ and
$\tilde{b}'\otimes\tilde{b}\in B^{3,3}\otimes B^{2,2}$ be the following elements:
\begin{align*}
b\otimes b'=
\Yvcentermath1\young(22,34)\otimes\young(113,234,455)\,,\qquad
\tilde{b}'\otimes\tilde{b}=\Yvcentermath1\young(122,334,445)\otimes\young(13,25)\,.
\end{align*}
Then we have $R:b\otimes b'\longmapsto\tilde{b}'\otimes\tilde{b}$ since we have
\begin{align*}
\Yvcentermath1\left(\young(113,234,455)\longleftarrow 3422\right)=
\left(\young(13,25)\longleftarrow 445334122\right)=
\young(11224,2333,445,5).
\end{align*}
In the tableau on the right hand side, we have one node which is outside of the shape
$\Yvcentermath1\Yboxdim6pt\yng(5,5,3)$.
Therefore we have $H(b\otimes b')=1$.
\qed
\end{example}

In general, let $p\in B=B^{r_1,s_1}\otimes B^{r_2,s_2}\otimes\cdots\otimes B^{r_m,s_m}$.
Suppose that $B'$ is an arbitrary reordering of the rectangles of $B$.
Then we obtain a map $R:p\longmapsto p'\in B'$ by successive applications of the local combinatorial $R$-matrices.

\subsection{Time evolutions}
Following the papers \cite{HHIKTT,FOY}, we introduce a crystal basis formulation
of box-ball systems.
The states of the box-ball system are elements of tensor products of crystals, which we call paths.
Suppose that we are given the path:
\begin{align}\label{eq:path}
p=b_1\otimes b_2\otimes\cdots\otimes b_m\in
B^{r_1,s_1}\otimes B^{r_2,s_2}\otimes\cdots\otimes B^{r_m,s_m}.
\end{align}
In the following, we define the time evolutions $T^{r,s}$
($r\in I_0$, $s\in\mathbb{Z}_{>0}$) of the box-ball system on $p$.

The operator $T^{r,s}$ depends on the choice of the highest weight element $u^{r,s}\in B^{r,s}$.
Here $u^{r,s}$ is given explicitly as follows:
\begin{align}\label{eq:def_carrier}
u^{r,s}:=
\overbrace{
\begin{array}{|c|c|c|c|c|}
\hline
1&1&1&\cdots&1\\
\hline
2&2&2&\cdots&2\\
\hline
\vdots&\vdots&\vdots&\ddots&\vdots\\
\hline
r&r&r&\cdots&r\\
\hline
\end{array}
}^s
\in B^{r,s}.
\end{align}
Let us express the isomorphism $R:b\otimes b'\longmapsto\tilde{b}'\otimes\tilde{b}$
by the following vertex diagram:
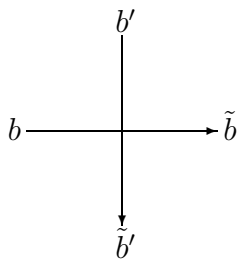
\[
\unitlength 12pt
\begin{picture}(8,8)
\put(1,4){\vector(1,0){6}}
\put(4,7){\vector(0,-1){6}}
\put(0.4,3.7){$b$}
\put(3.8,7.1){$b'$}
\put(3.8,0){$\tilde{b}'$}
\put(7.2,3.7){$\tilde{b}$}
\end{picture}
\]

For the selected values of $r$ and $s$, we set $u^{(1)}:=u^{r,s}$
and consider the following diagram:
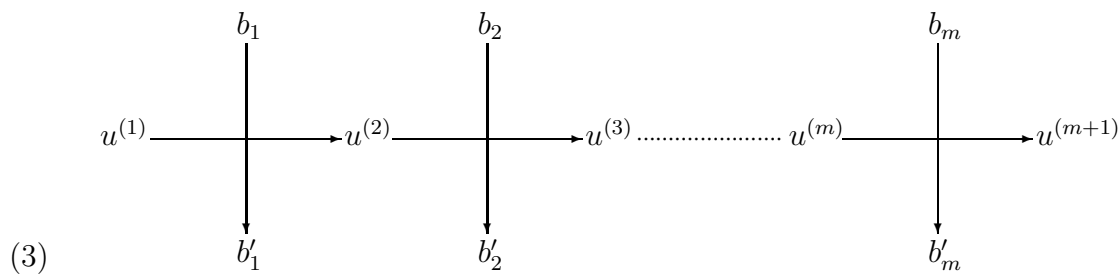
\begin{equation}\label{eq:def_bbs}
\unitlength 12pt
\begin{picture}(31,8)
\put(1,4){\vector(1,0){6}}
\put(4,7){\vector(0,-1){6}}
\put(-0.6,3.8){$u^{(1)}$}
\put(3.7,7.3){$b_1$}
\put(3.7,0.1){$b_1'$}
\put(7.1,3.8){$u^{(2)}$}
\put(7.6,0){
\put(1,4){\vector(1,0){6}}
\put(4,7){\vector(0,-1){6}}
\put(3.7,7.3){$b_2$}
\put(3.7,0.1){$b_2'$}
\put(7.1,3.8){$u^{(3)}$}
}
\multiput(16.4,4)(0.2,0){23}{\circle*{0.1}}
\put(21.8,0){
\put(1,4){\vector(1,0){6}}
\put(4,7){\vector(0,-1){6}}
\put(-0.7,3.8){$u^{(m)}$}
\put(3.7,7.3){$b_m$}
\put(3.7,0.1){$b_m'$}
\put(7.1,3.8){$u^{(m+1)}$}
}
\end{picture}
\end{equation}
More precisely, we have $R:u^{(1)}\otimes b_1\longmapsto b'_1\otimes u^{(2)}$,
and by using $u^{(2)}$, we compute $R:u^{(2)}\otimes b_2\longmapsto b'_2\otimes u^{(3)}$,
and continue similarly.

Then, by using the bottom row of (\ref{eq:def_bbs}),
we define the time evolutions $T^{r,s}$ of the box-ball system as
\[
T^{r,s}(p):=b'_1\otimes b'_2\otimes\cdots\otimes b'_m.
\]
Thus we obtain an infinite set of time evolutions $T^{r,s}$
depending on the parameters $r$ and $s$.
By virtue of the Yang--Baxter relation of the combinatorial $R$-matrices,
we can show
\[
T^{r_2,s_2}T^{r_1,s_1}(p)=T^{r_1,s_1}T^{r_2,s_2}(p)
\]
for arbitrary tensor products $p$, which implies that the box-ball system is integrable.
See \cite{FOY} for details.

\begin{example}
The example given at the beginning of Section \ref{sec:bbs}
corresponds to the case $p\in (B^{1,1})^{\otimes 30}$.
In this case, we have $T^{1,3}(p)=T^{1,4}(p)=\cdots =:T^{1,\infty}(p)$.
\qed
\end{example}

The following quantity was originally introduced in \cite{FOY}.

\begin{definition}\label{def:energy}
For the path $p$ of (\ref{eq:path}), we define
\begin{align*}
E^{r,s}(p):=\sum_{k=1}^m H(u^{(k)}\otimes b_k)
\end{align*}
where $u^{(1)}=u^{r,s}$ and the remaining elements $u^{(k)}$ are defined
as in the diagram (\ref{eq:def_bbs}).
\end{definition}

\section{Rigged configuration bijection}
\label{sec:rig}
Rigged configurations are combinatorial objects which are in one to one
correspondence with elements of the tensor products of crystals.
Originally, rigged configurations were introduced in \cite{KKR,KR}
with a motivated by the Bethe ansatz for the isotropic
Heisenberg model \cite{Bethe}.
Although there are serious technical difficulties concerning the Bethe ansatz,
recent progress \cite{NW,KS2} has enabled us to make the relationship
between rigged configurations and solutions to the so-called Bethe ansatz equations
explicit \cite{KS} (see \cite{DG} for a recent follow-up).
Nowadays, rigged configurations are regarded as a canonical presentation
of the affine crystals.
For example, several highly non-trivial involutions of the underlying algebras have
beautiful relationship with rigged configurations \cite{OS,OSS}.

In Section \ref{sec:definitions}, we prepare several basic definitions
about the rigged configurations which we need in both Sections \ref{sec:RC} and \ref{sec:Phi}.
In Section \ref{sec:RC}, we define rigged configurations
and in Section \ref{sec:Phi} we describe the bijection between
the tensor product of crystals and the rigged configurations.
Sections \ref{sec:RC} and \ref{sec:Phi} are mostly independent with each other.

\subsection{Basic definitions}
\label{sec:definitions}
In this paper, we consider rigged configurations corresponding
to tensor products of the form $\bigotimes_i B^{1,s_i}$ of type $\hat\sl_n$.
However, we remark that rigged configurations have a natural
generalization which includes arbitrary tensor products $\bigotimes_i B^{r_i,s_i}$ \cite{KSS}.

We denote a rigged configuration in our case as follows:
\begin{align*}
(\nu,J)=\left\{
(\nu^{(0)}),(\nu^{(1)},J^{(1)}), (\nu^{(2)},J^{(2)}),\ldots,(\nu^{(n-1)},J^{(n-1)})
\right\}.
\end{align*}
Here, for $a=0,1,\ldots,n-1$, $\nu^{(a)}=(\nu^{(a)}_1,\nu^{(a)}_2,\ldots)$
are partitions called the {\bf configurations}.
For each part of the configuration $\nu^{(a)}_i$,
we associate an integer $J^{(a)}_i$ called the {\bf rigging}.
We call the pair $(\nu^{(a)}_i,J^{(a)}_i)$ a {\bf string}.
For $a\in I_0$ and $k\in\mathbb{Z}_{\geq 0}$,
define the {\bf vacancy number} as
\begin{align*}
P^{(a)}_k(\nu)&=Q_k(\nu^{(a-1)})-2Q_k(\nu^{(a)})+Q_k(\nu^{(a+1)}),\\
Q_k(\nu^{(a)})&=\sum_{i}\min(k,\nu^{(a)}_i).
\nonumber
\end{align*}
Here we understand that $\nu^{(n)}=\emptyset$.

In the following, we assume that every $(\nu^{(a)},J^{(a)})$ contains the string $(0,0)$.
Note that we always have $P^{(a)}_0(\nu)=0$.

\subsection{Rigged configurations}
\label{sec:RC}
On the set of all possible $(\nu,J)$, we have to impose several conditions to specify valid rigged configurations.
For this purpose we first explicitly identify the rigged configurations corresponding to highest weight vectors,
and then use the Kashiwara operators introduced by Schilling \cite{Sc_IMRN} to generate the rest of the rigged configurations.  In \cite[Section 3]{S:2014} it is shown that these operators satisfy the axioms of crystals directly from their definition and the definitions of rigged configurations.

If all the strings $(\nu^{(a)}_i,J^{(a)}_i)$ of $(\nu,J)$ satisfy the condition
\begin{align*}
P^{(a)}_{\nu^{(a)}_i}(\nu)\geq J^{(a)}_i\geq 0,
\end{align*}
then $(\nu,J)$ is the rigged configuration corresponding to a highest weight vector.  For the string $(\nu^{(a)}_i,J^{(a)}_i)$, the quantity
$P^{(a)}_{\nu^{(a)}_i}(\nu)-J^{(a)}_i$ is called the {\bf corigging}.

\begin{definition}
Let $x_\ell$ be the smallest rigging of $(\nu^{(a)},J^{(a)})$.
\begin{enumerate}
\item[(1)]
If $x_\ell\geq 0$, define $\tilde{e}_a(\nu,J)=0$.
Otherwise let $\ell$ be the minimal length of the strings of $(\nu^{(a)},J^{(a)})$
with the rigging $x_\ell$.
Then $\tilde{e}_a(\nu,J)$ is obtained by replacing one of the strings $(\ell,x_\ell)$
by $(\ell-1,x_\ell+1)$ while changing all the other riggings to keep the coriggings fixed.
\item[(2)]
Let $\ell$ be the maximal length of the strings of $(\nu^{(a)},J^{(a)})$
which have the rigging $x_\ell$.
Then $\tilde{f}_a(\nu,J)$ is obtained by replacing one of the strings $(\ell,x_\ell)$
by $(\ell+1,x_\ell-1)$ while changing all the other riggings to keep the coriggings fixed.
If the new rigging exceeds the corresponding new vacancy number,
redefine $\tilde{f}_a(\nu,J)=0$.
\end{enumerate}
\end{definition}
The set of all rigged configurations is then obtained by applying
$\tilde{f}_a$ ($a\in I_0$) on the highest weight rigged configurations in all possible ways.

Here some remarks are in order.
In the theory of rigged configurations, the ordering of rows of $\nu^{(a)}$ is not essential.
Thus we can think of a rigged configuration as a set of strings.  Also we remark that the partition $\nu^{(0)}$ of $(\nu,J)$ specifies the shape
of the corresponding tensor product.
More precisely, if we consider $\bigotimes_iB^{1,s_i}$,
then we have $\nu^{(0)}=(s_1,s_2,\ldots,s_m)$.
\begin{example}
Let us consider the following rigged configuration $(\nu,J)$:
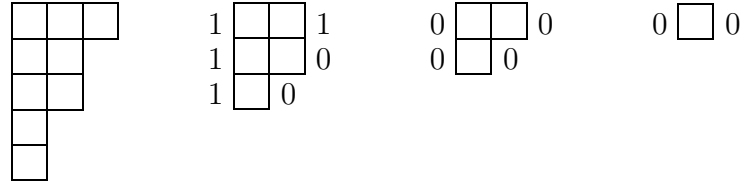
\begin{center}
\unitlength 12pt
\begin{picture}(23,6)
\put(0,0){\yng(3,2,2,1,1)}
\put(7,0){
\put(-0.8,4.6){1}
\put(-0.8,3.5){1}
\put(-0.8,2.4){1}
\put(0,2.24){\yng(2,2,1)}
\put(2.6,4.6){1}
\put(2.6,3.5){0}
\put(1.5,2.4){0}}
\put(14,0){
\put(-0.8,4.6){0}
\put(-0.8,3.5){0}
\put(0,3.35){\yng(2,1)}
\put(2.6,4.6){0}
\put(1.5,3.5){0}}
\put(21,0){
\put(-0.8,4.6){0}
\put(0,4.46){\yng(1)}
\put(1.5,4.6){0}}
\end{picture}
\end{center}
Here the partitions correspond to $\nu^{(0)},\ldots,\nu^{(3)}$
from left to right.
For $\nu^{(1)},\nu^{(2)}$ and $\nu^{(3)}$, we put the vacancy number (resp. rigging)
to left (resp. right) of the corresponding row of the diagram.
Note that $(\nu,J)$ is highest weight element.
$\tilde{f}_1(\nu,J)$ is the following rigged configuration:
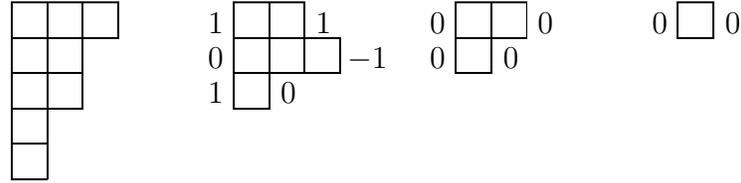
\begin{center}
\unitlength 12pt
\begin{picture}(23,6)
\put(0,0){\yng(3,2,2,1,1)}
\put(7,0){
\put(-0.8,4.6){1}
\put(-0.8,3.5){0}
\put(-0.8,2.4){1}
\put(0,2.24){\yng(2,3,1)}
\put(2.6,4.6){1}
\put(3.6,3.5){$-1$}
\put(1.5,2.4){0}}
\put(14,0){
\put(-0.8,4.6){0}
\put(-0.8,3.5){0}
\put(0,3.35){\yng(2,1)}
\put(2.6,4.6){0}
\put(1.5,3.5){0}}
\put(21,0){
\put(-0.8,4.6){0}
\put(0,4.46){\yng(1)}
\put(1.5,4.6){0}}
\end{picture}
\end{center}
Let us compute $\tilde{f}_2(\nu,J)$.
First we obtain the following object:
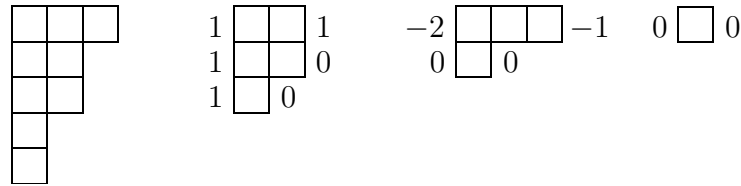
\begin{center}
\unitlength 12pt
\begin{picture}(23,6)
\put(0,0){\yng(3,2,2,1,1)}
\put(7,0){
\put(-0.8,4.6){1}
\put(-0.8,3.5){1}
\put(-0.8,2.4){1}
\put(0,2.24){\yng(2,2,1)}
\put(2.6,4.6){1}
\put(2.6,3.5){0}
\put(1.5,2.4){0}}
\put(14,0){
\put(-1.6,4.6){$-2$}
\put(-0.8,3.5){$0$}
\put(0,3.35){\yng(3,1)}
\put(3.6,4.6){$-1$}
\put(1.5,3.5){$0$}}
\put(21,0){
\put(-0.8,4.6){0}
\put(0,4.46){\yng(1)}
\put(1.5,4.6){0}}
\end{picture}
\end{center}
However $(\nu^{(2)},J^{(2)})$ contains the string $(3,-1)$
and its rigging exceeds the corresponding vacancy number $-2$.
Therefore we have $\tilde{f}_2(\nu,J)=0$.
Finally, $\tilde{f}_3(\nu,J)$ is the following:
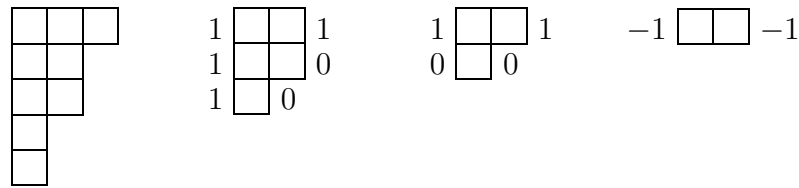
\begin{center}
\unitlength 12pt
\begin{picture}(23,6)
\put(0,0){\yng(3,2,2,1,1)}
\put(7,0){
\put(-0.8,4.6){1}
\put(-0.8,3.5){1}
\put(-0.8,2.4){1}
\put(0,2.24){\yng(2,2,1)}
\put(2.6,4.6){1}
\put(2.6,3.5){0}
\put(1.5,2.4){0}}
\put(14,0){
\put(-0.8,4.6){1}
\put(-0.8,3.5){0}
\put(0,3.35){\yng(2,1)}
\put(2.6,4.6){1}
\put(1.5,3.5){0}}
\put(21,0){
\put(-1.6,4.6){$-1$}
\put(0,4.46){\yng(2)}
\put(2.6,4.6){$-1$}}
\end{picture}
\end{center}

In the next section, we will define the bijection between rigged configurations and elements of tensor products of crystals.
Under this bijection, $(\nu,J)$ corresponds to
$\Yvcentermath1\young(111)\otimes\young(22)\otimes\young(13)\otimes\young(4)\otimes\young(3)$.
\qed
\end{example}

\subsection{The rigged configuration bijection}
\label{sec:Phi}
One of the most important constituents of the rigged configuration theory is
a bijection between tensor products of crystals and rigged configurations
\begin{align*}
\Phi: B^{1,s_1}\otimes B^{1,s_2}\otimes\cdots\otimes B^{1,s_m}
\longmapsto\left\{(\nu,J)\,\bigr|\,\nu^{(0)}=(s_1,s_2,\ldots,s_m)\right\}.
\end{align*}
We call $\Phi$ the {\bf rigged configuration bijection}.
We remark that our bijection is a special case of a more general bijection defined for
$B^{r_1,s_1}\otimes B^{r_2,s_2}\otimes\cdots\otimes B^{r_m,s_m}$ \cite{KSS}.

In order to define the rigged configuration bijection,
we introduce an important notion about strings.
If the string $(\nu^{(a)}_i,J^{(a)}_i)$ satisfies the condition
$P^{(a)}_{\nu^{(a)}_i}(\nu)=J^{(a)}_i$ (i.e., if the corresponding corigging is 0),
we say the string is {\bf singular}.
We remark that we always have the relation $P^{(a)}_{\nu^{(a)}_i}(\nu)\geq J^{(a)}_i$
by construction of the rigged configuration in Section \ref{sec:RC}.

Suppose that we are given
\begin{align*}
p=b_1\otimes b_2\otimes\cdots\otimes b_m\in
B^{1,s_1}\otimes B^{1,s_2}\otimes\cdots\otimes B^{1,s_m}.
\end{align*}
The construction of $\Phi(p)$ is a recursive procedure which proceeds from left to right of $p$.
To begin with, we suppose that the empty path corresponds to the
empty rigged configuration.
Suppose that we have constructed
$\Phi(b_1\otimes b_2\otimes\cdots\otimes b_{i-1})$
for some $1\leq i-1<m$.
Suppose that the semistandard tableau representation of $b_i$ is
\begin{align*}
b_i=
\begin{array}{|c|c|c|c|}
\hline
c_{s_i} & \cdots & c_2 &c_1\\
\hline
\end{array}.
\end{align*}
Again, the procedure $\Phi$ is recursive starting from the letter $c_1$.
Suppose that we have constructed up to the letter $c_{k-1}$
and obtained the rigged configuration
\begin{align*}
&\Phi(b_1\otimes\cdots\otimes b_{i-1}\otimes
\begin{array}{|c|c|c|c|}
\hline
c_{k-1} & \cdots & c_2 &c_1\\
\hline
\end{array}
)\\
&=
(\eta,I)=\left\{
(\eta^{(0)}),(\eta^{(1)},I^{(1)}), (\eta^{(2)},I^{(2)}),\ldots,(\eta^{(n-1)},I^{(n-1)})
\right\}.
\end{align*}
Then the next rigged configuration
\[
(\eta',I'):=
\Phi(b_1\otimes\cdots\otimes b_{i-1}\otimes
\begin{array}{|c|c|c|c|}
\hline
c_{k} & \cdots & c_2 &c_1\\
\hline
\end{array}
)
\]
is obtained by the following procedure.
Recall that we assume that $(\eta^{(a)},J^{(a)})$ ($a\in I_0$) contains the empty singular string $(0,0)$.

{\sf Main Algorithm.}
\begin{enumerate}
\item
Look for the largest singular string of $(\eta^{(c_k-1)},I^{(c_k-1)})$.
Let us denote the length of the string by $\ell^{(c_k-1)}$.
\item
Recursively, we do the following procedure for $a=c_k-1, c_k-2, \ldots,2$.
Suppose that we have determined $\ell^{(a)}$.
Then we look for the largest singular string of $(\eta^{(a-1)},I^{(a-1)})$
whose length does not exceed $\ell^{(a)}$.
We denote the length of the latter string by $\ell^{(a-1)}$.
\item
$\eta'$ is obtained by adding one box to one of the singular strings of $(\eta^{(a)},I^{(a)})$ with length $\ell^{(a)}$, for $a\in I_0$.
For $\eta^{(0)}$, we add one box to a length $k-1$ row.
\item
$I'$ is defined in the following way.
If the corresponding row is not changed from $\eta$,
we do not modify the rigging.
Otherwise we choose the new rigging so that the new string becomes singular in $(\eta',I')$.
\end{enumerate}

If we reverse the above procedure, we obtain the algorithm for
the inverse map $\Phi^{-1}$ (see, for example, \cite[Appendix A]{S} for a description
and an example).
A Mathematica implementation of the rigged configuration
bijection is available at the web page \cite{S:web}.

\begin{example}
Let us consider
$p=\Yvcentermath1\young(2)\otimes\young(24)\otimes\young(3)$.
Then the rigged configuration bijection proceeds as follows.
\begin{center}
\unitlength 12pt
\begin{tabular}{rl}
\raisebox{4pt}[0pt][10pt]{
$\Phi(\Yvcentermath1\young(2))=$}&
\begin{picture}(6,1.5)
\put(0,0){\yng(1)}
\put(5,0){
\put(-1.5,0.2){$-1$}
\put(0,0){\yng(1)}
\put(1.3,0.2){$-1$}}
\end{picture}\\
\raisebox{17pt}[0pt][10pt]{
$\Phi(\Yvcentermath1\young(2)\otimes\young(4))=$}&
\begin{picture}(17,2.5)
\put(0,0){\yng(1,1)}
\put(5,0){
\put(-1.5,1.3){$-1$}
\put(-1.5,0.2){$-1$}
\put(0,0){\yng(1,1)}
\put(1.3,1.3){$-1$}
\put(1.3,0.2){$-1$}}
\put(10.5,1.1){
\put(-0.7,0.2){1}
\put(0,0){\yng(1)}
\put(1.4,0.2){1}}
\put(16,1.1){
\put(-1.5,0.2){$-1$}
\put(0,0){\yng(1)}
\put(1.3,0.2){$-1$}}
\end{picture}\\
\raisebox{17pt}[0pt][10pt]{
$\Phi(\Yvcentermath1\young(2)\otimes\young(24))=$}&
\begin{picture}(17,2.5)
\put(0,0){\yng(1,2)}
\put(5,0){
\put(-1.5,1.3){$-2$}
\put(-1.5,0.2){$-1$}
\put(0,0){\yng(2,1)}
\put(2.4,1.3){$-2$}
\put(1.3,0.2){$-1$}}
\put(10.5,1.1){
\put(-0.7,0.2){1}
\put(0,0){\yng(1)}
\put(1.4,0.2){1}}
\put(16,1.1){
\put(-1.5,0.2){$-1$}
\put(0,0){\yng(1)}
\put(1.3,0.2){$-1$}}
\end{picture}\\
\raisebox{30pt}[0pt][10pt]{
$\Phi(\Yvcentermath1\young(2)\otimes\young(24)\otimes\young(3))=$}&
\begin{picture}(18,3.6)
\put(0,0){\yng(1,2,1)}
\put(5,1.1){
\put(-1.5,1.3){$-2$}
\put(-1.5,0.2){$-2$}
\put(0,0){\yng(2,2)}
\put(2.4,1.3){$-2$}
\put(2.4,0.2){$-2$}}
\put(10.5,2.2){
\put(-0.7,0.2){1}
\put(0,0){\yng(2)}
\put(2.5,0.2){1}}
\put(16,2.2){
\put(-1.5,0.2){$-1$}
\put(0,0){\yng(1)}
\put(1.3,0.2){$-1$}}
\end{picture}
\end{tabular}
\end{center}
\qed
\end{example}

Recall that general rigged configurations are defined by applying the Kashiwara operators $\tilde{f}_a$ ($a\in I_0$) to highest-weight rigged configurations.
Concerning the consistency with the rigged configuration bijection $\Phi$,
we have the following fundamental property \cite{DS,S:2014}.

\begin{theorem}
For $a\in I_0$, the Kashiwara operators $\tilde{e}_a$ and $\tilde{f}_a$
commute with the rigged configuration bijection $\Phi$:
$[\tilde{e}_a,\Phi]=[\tilde{f}_a,\Phi]=0$.
\end{theorem}

Recall that we identify rigged configurations if the only difference
is the ordering of the rows of $\nu^{(0)}$.
On the other hand, if we swap positions of rectangles in a tensor product
of crystals, we need to apply the combinatorial $R$-matrix.
The following result \cite{KSS} is a very deep and important property
of the rigged configuration bijection $\Phi$.
\begin{theorem}\label{th:combR}
Let $p\in B^{r_1,s_1}\otimes B^{r_2,s_2}\otimes\cdots\otimes B^{r_m,s_m}$
and let $B'$ be an arbitrary reordering of rectangles of $B$.
If $R:p\longmapsto p'\in B'$, we have $\Phi(p)=\Phi(p')$.
\end{theorem}

We have constructed the rigged configuration bijection in a combinatorial language.
However, an important observation is that the algorithm $\Phi$ itself is an algebraic object in nature.
More precisely, we have the following fundamental result.
Recall that we defined the quantity $E^{r,s}$ in Definition \ref{def:energy}.
\begin{theorem}[\cite{S}]\label{th:E=Q}
Let $p\in B^{r_1,s_1}\otimes B^{r_2,s_2}\otimes\cdots\otimes B^{r_m,s_m}$
and let $(\nu,J)=\Phi(p)$.
Then we have
\begin{align}\label{eq:E=Q}
E^{r,s}(p)=Q_s(\nu^{(r)}).
\end{align}
Recall that $Q_s(\nu^{(r)})$ is the number of boxes in the first $s$ columns of $\nu^{(r)}$.
\end{theorem}

This formula plays a crucial role in the present paper.
Note that even if we consider the special case $p\in\bigotimes_iB^{1,s_i}$
(as in the present paper), we still need to consider the general
time evolutions $T^{r,s}$.

We remark that in \cite{S}, it is claimed that the relation (\ref{eq:E=Q}) recovers
the combinatorial procedure of $\Phi$.\footnote{More precisely,
the relation (\ref{eq:E=Q}) provides information for entire columns
of each $B^{r,s}$ and does not provide the data for each box of the column.
However, it is shown that this information contains enough data
to determine the corresponding rigged configuration.}
Therefore, not only the configurations $\nu^{(a)}$
but also the procedure $\Phi$ has an algebraic origin.
However, in order to determine the riggings $J^{(a)}$ from the energy functions,
we still need a purely combinatorial procedure.
Thus the origin of the riggings is yet unclear.

\section{Loop symmetric functions}
\label{sec:loop}
For most of the details of this section we refer the reader to \cite{LPloop,L,LP}.  Following \cite{LP}, we use the font $\x_i$ to denote variables in the ``geometric" or ``rational" world, distinguishing such variables from the variables $x_i$ in the ``tropical" or ``combinatorial" world.
Also the upper indices such as $(k)$ in $\x^{(k)}$ are always taken modulo $n$.
For the correspondence between the following constructions and the crystal bases,
see Section \ref{sec:tropical}.

\subsection{Birational $R$-matrices and loop-symmetric functions}
\label{sec:birationalR}
The conventions in this section are identical to those in \cite{LP}.
For two $n$-tuples $(\x^{(1)},\ldots,\x^{(n)})$ and $(\y^{(1)},\ldots,\y^{(n)})$ of variables and $r \in \Z/n\Z$, denote
$$
\kappa_r(\x,\y) = \sum_{s=0}^{n-1} \prod_{t=1}^s \y^{(r+t)}\; \prod_{t=s+1}^{n-1} \x^{(r+t)}.
$$
For example, for $n=4$, one has
\begin{align*}
\kappa_1(\x,\y) &= \x^{(2)}\x^{(3)}\x^{(4)}+\y^{(2)}\x^{(3)}\x^{(4)}+\y^{(2)}\y^{(3)}\x^{(4)}+\y^{(2)}\y^{(3)}\y^{(4)},\\
\kappa_2(\x,\y) &= \x^{(3)}\x^{(4)}\x^{(1)}+\y^{(3)}\x^{(4)}\x^{(1)}+\y^{(3)}\y^{(4)}\x^{(1)}+\y^{(3)}\y^{(4)}\y^{(1)},\\
\kappa_3(\x,\y) &= \x^{(4)}\x^{(1)}\x^{(2)}+\y^{(4)}\x^{(1)}\x^{(2)}+\y^{(4)}\y^{(1)}\x^{(2)}+\y^{(4)}\y^{(1)}\y^{(2)},\\
\kappa_4(\x,\y) &= \x^{(1)}\x^{(2)}\x^{(3)}+\y^{(1)}\x^{(2)}\x^{(3)}+\y^{(1)}\y^{(2)}\x^{(3)}+\y^{(1)}\y^{(2)}\y^{(3)}.
\end{align*}
Now define a rational map 
$$
\R: \Q(\x^{(1)},\ldots,\x^{(n)},\y^{(1)},\ldots,\y^{(n)}) \to \Q(\x^{(1)},\ldots,\x^{(n)},\y^{(1)},\ldots,\y^{(n)})$$ 
by
\begin{equation}\label{E:s}
\R(\x^{(i)})=\y^{(i+1)}\frac{\kappa_{i+1}(\x,\y)}{\kappa_i(\x,\y)} \ \ \ \ \text{and} \ \ \ \  \R(\y^{(i)})=\x^{(i-1)}\frac{\kappa_{i-1}(\x,\y)}{\kappa_i(\x,\y)}.
\end{equation}
We call $\R$ the {\bf {birational $R$-matrix}}, or the geometric $R$-matrix. 

If we are given $m$ sets of variables $\{(\x_j^{(1)},\ldots,\x_j^{(n)}) \mid 1 \leq j \leq m\}$, then we let $\R_k$ denote the rational map acting on $\Q(\x^{(i)}_j)$, 
fixing the sets of variables $\x_j$ for $j \neq k, k+1$, and then applying $\R$ to the variables $\x_k, \x_{k+1}$.

Define the {\bf {elementary loop symmetric functions}} $e_k^{(r)}$ as follows:
\begin{align*}
e_k^{(r)}(\x_1,\x_2,\ldots,\x_m) &= \sum_{1 \leq i_1 < i_2 < \cdots < i_k\leq m} \x_{i_1}^{(r)} \x_{i_2}^{(r+1)} \cdots \x_{i_k}^{(r+k-1)}.
\end{align*}
For example, take $n = 2$ and $m = 3$.  Then
$$
\begin{array}{ll}
e_1^{(1)} = \x_1^{(1)}+\x_2^{(1)}+\x_3^{(1)}  &e_1^{(2)} = \x_1^{(2)}+\x_2^{(2)}+\x_3^{(2)}\\
e_2^{(1)} =\x_1^{(1)}\x_2^{(2)}+\x_2^{(1)}\x_3^{(2)}+\x_1^{(1)}\x_3^{(2)} & 
e_2^{(2)}=\x_1^{(2)}\x_2^{(1)}+\x_1^{(2)}\x_3^{(1)}+\x_2^{(2)}x_3^{(1)}\\
e_3^{(1)}=\x_1^{(1)}\x_2^{(2)}\x_3^{(1)} & e_3^{(2)}=\x_1^{(2)}\x_2^{(1)}\x_3^{(2)}.
\end{array}
$$
By convention, $e_k^{(r)} = 0$ for $k < 0$, and $e_0^{(r)}=1$.  We call the upper index the \defn{color}.  
We denote the subring of  $\Z[\x_1^{(1)},\x_1^{(2)},\ldots,\x_1^{(n)},\x_2^{(1)},\ldots,\x_m^{(n)}]$ generated by the loop elementary symmetric polynomials $e_i^{(k)}$ by 
$\LSym_m$. We call it the ring of {\bf{loop symmetric polynomials}} in $m$ sets of variables.  This ring was introduced in \cite{LPloop} in the context of the theory of total positivity of loop groups.

\begin{theorem}[{\cite[Theorem 4.3]{L}, see also \cite{Y}}]
\label{T:R}
\
\begin{enumerate}
\item
The rational maps $\R_j$ generate a birational action of $S_m$ on $\Q(\x^{(i)}_j)$.
\item
The loop elementary symmetric functions, and thus the ring $\LSym_m$ they generate, are invariants of this action.
\end{enumerate}
\end{theorem}

In fact, in \cite[Theorem 4.4]{L} it was announced that {\it {all}} polynomial invariants of this $S_m$ action lie in $\LSym_m$.
\begin{theorem}[Lam-Pylyavskyy, unpublished]
\label{thm:all}
We have
$$
\LSym_m \simeq \Q(\x^{(i)}_j)^{S_m} \cap \Q[\x^{(i)}_j].
$$
\end{theorem}

\subsection{Loop Schur functions and cylindric loop Schur functions}

A square $s = (i,j)$ in the $i$-th row (counting from top to bottom) and $j$-th column (counting from left to right) has {\bf{content}} $c(s)=i-j$.
Our notion of content is the negative of the usual one.

Let $\rho/\nu$ be a skew shape.
A {\bf semistandard Young tableau} $T$ of shape
$\rho/\nu$ is a filling of each square $s \in \rho/\nu$ with an
integer $T(s) \in \Z_{> 0}$ so that the rows are weakly-increasing and the columns are strictly increasing.  
For $r\in \Z/n\Z$, the {\bf {$r$-weight}} $\x^{\wt^{(r)}(T)}$ of a 
tableaux $T$ is
\begin{align}\label{def:r-weight}
\x^{\wt^{(r)}(T)} = \prod_{s \in \rho/\nu}\x_{T(s)}^{(c(s)+r)}.
\end{align}

We shall draw our shapes and tableaux in English notation.
Then the definition (\ref{def:r-weight}) gives the following correspondence:
\begin{align}\label{eq:tableau_weight}
\begin{array}{|c|c|c|c|c}
\hline
T(1,1)&T(1,2)&T(1,3)&T(1,4)&\cdots\\
\hline
T(2,1)&T(2,2)&T(2,3)&T(2,4)&\cdots\\
\hline
T(3,1)&T(3,2)&T(3,3)&T(3,4)&\cdots\\
\hline
T(4,1)&T(4,2)&T(4,3)&T(4,4)&\cdots\\
\hline
\vdots&\vdots&\vdots&\vdots&\ddots
\end{array}
\longmapsto
\begin{array}{|c|c|c|c|c}
\hline
\x^{(r)}_{T(1,1)}&\x^{(-1+r)}_{T(1,2)}&\x^{(-2+r)}_{T(1,3)}&\x^{(-3+r)}_{T(1,4)}&\cdots\\
\hline
\x^{(1+r)}_{T(2,1)}&\x^{(r)}_{T(2,2)}&\x^{(-1+r)}_{T(2,3)}&\x^{(-2+r)}_{T(2,4)}&\cdots\\
\hline
\x^{(2+r)}_{T(3,1)}&\x^{(1+r)}_{T(3,2)}&\x^{(r)}_{T(3,3)}&\x^{(-1+r)}_{T(3,4)}&\cdots\\
\hline
\x^{(3+r)}_{T(4,1)}&\x^{(2+r)}_{T(4,2)}&\x^{(1+r)}_{T(4,3)}&\x^{(r)}_{T(4,4)}&\cdots\\
\hline
\vdots&\vdots&\vdots&\vdots&\ddots
\end{array}
\end{align}
For example, let $n=3$. Then the $1$-weight of the following tableau is computed as
\[
T=
\tableau[mY]{
\bl \circ&\bl \circ&\bl \circ&1&1&1&3\\
\bl \circ&1&2&2&3&4\\
\bl \circ&3&3&4}
\longmapsto
\tableau[mY]{
\bl \circ&\bl \circ&\bl \circ&\x_1^{(1)}&\x_1^{(3)}&\x_1^{(2)}&\x_3^{(1)}\\
\bl \circ&\x_1^{(1)}&\x_2^{(3)}&\x_2^{(2)}&\x_3^{(1)}&\x_4^{(3)}\\
\bl \circ&\x_3^{(2)}&\x_3^{(1)}&\x_4^{(3)}}
\]
Thus $\x^{\wt^{(1)}(T)}=(\x_1^{(1)})^2 (\x_3^{(1)})^3 \x_1^{(2)} \x_2^{(2)} \x_3^{(2)} \x_1^{(3)} \x_2^{(3)} (\x_4^{(3)})^2$.
We define the {\bf {loop skew Schur function}} by
$$
s^{(r)}_{\lambda/\mu}({\x}) = \sum_{T} \x^{\wt^{(r)}(T)}
$$
where the summation is over all semistandard Young tableaux of
(skew) shape $\lambda/\mu$ using the letters $1,2,\ldots,m$.

\begin{example}\label{ex:standard_LSchur}
Let $n = 3$ and $m=3$.
Then we have
\begin{align*}
s^{(1)}_{2,1}(\x_1,\x_2,\x_3) = &\,\, \x_1^{(1)}\x_1^{(3)}\x_2^{(2)} + \x_1^{(1)}\x_2^{(3)}\x_2^{(2)} + \x_1^{(1)}\x_2^{(3)}\x_3^{(2)} +
\x_1^{(1)}\x_3^{(3)}\x_2^{(2)} + \\
&\,\,\x_1^{(1)}\x_1^{(3)}\x_3^{(2)} +
\x_2^{(1)}\x_2^{(3)}\x_3^{(2)} +
\x_1^{(1)}\x_3^{(3)}\x_3^{(2)} +
\x_2^{(1)}\x_3^{(3)}\x_3^{(2)},
\end{align*}
corresponding to the tableaux
$$
\tableau[sY]{1&1\\2} \qquad \tableau[sY]{1&2\\2} \qquad \tableau[sY]{1&2\\3} \qquad \tableau[sY]{1&3\\2}
$$
$$
\tableau[sY]{1&1\\3} \qquad \tableau[sY]{2&2\\3} \qquad \tableau[sY]{1&3\\3} \qquad \tableau[sY]{2&3\\3}
$$
where the upper left corner has content 0, and so gives a $1$-weight with color $1$.  Setting $\x_i^{(1)}=\x_i^{(2)} =\x_i^{(3)} = \x_i$ gives the usual Schur polynomial $s_{2,1}(\x_1,\x_2,\x_3)$.
\qed
\end{example}

In our definition of cylindric loop Schur functions we stay close to \cite{M}. Define the {\bf {cylinder}} $\C_s$ to be the following quotient of integer lattice:
$$\C_s = \Z^2/(n-s,s)\Z.$$ In other words, $\C_s$ is the quotient of $\Z^2$ by the shift that sends $(a,b)$ into $(a+n-s,b+s)$. The set $\C_s$ inherits a natural partial order from that on $\Z^2$ given by the transitive closure of the cover relations $(a,b)<(a+1,b)$ and $(a,b)<(a,b-1)$.  A box in the $i$-th row and $j$-th column of the Young diagram has coordinates $(j,-i)$.

We define a {\bf {cylindric skew shape}} $D$ to be a finite convex subposet of $\C_s$.  We shall also often identify cylindric skew shapes $D$ with infinite skew shapes periodic under shifts by the vector $(n-s,s)$.  The content of the coordinate $(a,b)$ in $\Z^2$ is given by $-a-b$.  Note that the notion of color is still well-defined on $\C_s$ as a shift by $(n-s,s)$ preserves the content modulo $n$.

The notion of semistandard tableau extends in an obvious way to periodic fillings of cylindric shapes: a semistandard Young tableau of a cylindric skew shape $D$ is a map $T: D \to \Z_{>0}$ satisfying $T(a,b)\leq T(a+1,b)$ and $T(a,b)<T(a,b-1)$ whenever the corresponding boxes lie in $D$.
For a cylindric skew shape $D$, we define the {\bf {cylindric loop Schur function}} by
\begin{align*}
s^{(r)}_{D}({x}) = \sum_{T} x^{\wt^{(r)}(T)},
\end{align*}
where the summation is over all semistandard Young tableaux of cylindric skew shape $D$.  Note that in this definition we think of $D$ as a finite subset of $\C_s$ (otherwise, the monomial $x^{\wt^{(r)}(T)}$ would have infinite degree).  By convention, $s^{(r)}_D(x) = 1$ if $D$ is the empty set. 

\begin{theorem}\label{thm:cylindric}
Let $D$ be a cylindric skew shape.  Then the cylindric loop Schur function $s^{(r)}_{D}(\x_1,\x_2,\ldots,\x_m)$ lies in $\LSym_m$.
\end{theorem}
The proof of Theorem \ref{thm:cylindric} is given in Section \ref{sec:lsymproof}.

Let $\lambda/\mu$ be a skew partition with the longest row $\lambda_1 = n-s$.  Define $\D_s(\lambda/\mu)$ to be the set of squares obtained by periodically propagating $\lambda/\mu$ via a shift by $(n-s,s)$. Note that only for certain $\lambda/\mu$ is $\D_s(\lambda/\mu)$ a well-defined cylindric skew shape; however, every cylindric skew shape $D$ is equal to $\D_s(\lambda/\mu)$ for some skew partition $\lambda/\mu$.
 If $\lambda_1<n-s$, we have $s^{(r)}_{\D_s(\lambda/\mu)}=s^{(r)}_{\lambda/\mu}$. Therefore, the cylindric loop Schur function is a generalization of the loop Schur function.

\begin{example}
\label{ex:cylindric_Schur}
For $n=3$, $s=1$ and $\lambda = (2,1)$, the corresponding cylindric shape $\D_1(\lambda)$
in $\C_1$ is as follows:
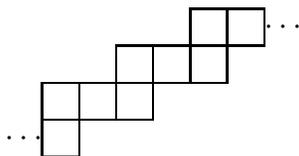
\[
\tableau[sY]{
\bl&\bl&\bl&\bl&\bl& & &\bl\cdots\\
\bl&\bl&\bl& & &\\
\bl& & &\\
\bl\cdots&}
\]
Here the shift is $(n-s,s)=(2,1)$.

Let $m=3$ as in Example \ref{ex:standard_LSchur}.
Then we have
\begin{align*}
s^{(1)}_{\D_1(2,1)}(\x_1,\x_2,\x_3) =
&\,\, \x_1^{(1)}\x_1^{(3)}\x_2^{(2)} + \x_1^{(1)}\x_2^{(3)}\x_2^{(2)} + \x_1^{(1)}\x_2^{(3)}\x_3^{(2)} + \\
&\,\,\x_1^{(1)}\x_1^{(3)}\x_3^{(2)} +
\x_2^{(1)}\x_2^{(3)}\x_3^{(2)} +
\x_1^{(1)}\x_3^{(3)}\x_3^{(2)} +
\x_2^{(1)}\x_3^{(3)}\x_3^{(2)}.
\end{align*}
For example, the first three tableaux listed in
Example \ref{ex:standard_LSchur} correspond to the following cylindric semistandard tableaux:
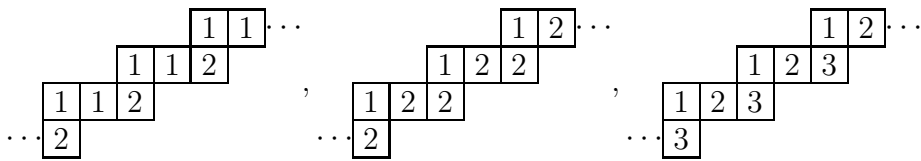
\[
\tableau[sY]{
\bl&\bl&\bl&\bl&\bl&1&1&\bl\cdots\\
\bl&\bl&\bl&1&1&2\\
\bl&1&1&2\\
\bl\cdots&2},
\tableau[sY]{
\bl&\bl&\bl&\bl&\bl&1&2&\bl\cdots\\
\bl&\bl&\bl&1&2&2\\
\bl&1&2&2\\
\bl\cdots&2},
\tableau[sY]{
\bl&\bl&\bl&\bl&\bl&1&2&\bl\cdots\\
\bl&\bl&\bl&1&2&3\\
\bl&1&2&3\\
\bl\cdots&3}.
\]
On the other hand, extending the fourth tableaux from Example \ref{ex:standard_LSchur} periodically we get
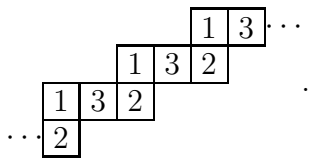
\[
\tableau[sY]{
\bl&\bl&\bl&\bl&\bl&1&3&\bl\cdots\\
\bl&\bl&\bl&1&3&2\\
\bl&1&3&2\\
\bl\cdots&2}.
\]
Since this is not a valid cylindrical semistandard tableau,
it does not contribute to $s^{(1)}_{\D_1(2,1)}$.
We can check that the rest of the tableaux do contribute to $s^{(1)}_{\D_1(2,1)}$.
\qed
\end{example}

\subsection{Proof of Theorem \ref{thm:cylindric}} \label{sec:lsymproof}

Consider a {\it {grid network}} $N(n,m)$ on a cylinder, consisting of $n$ directed wires running from one boundary to the other, without crossing, and $m$ directed wires forming closed loops around the cylinder,
crossing the former $n$ wires right to left. The case of $n=3$, $m=2$ is shown in Figure \ref{fig:rig1}. We assign variables $\x_i^{(j)}$ as {\it {weights}} to the crossings of wires 
(or {\it {vertices}}) as follows. 
\begin{figure}[h!]
    \begin{center}
    \input{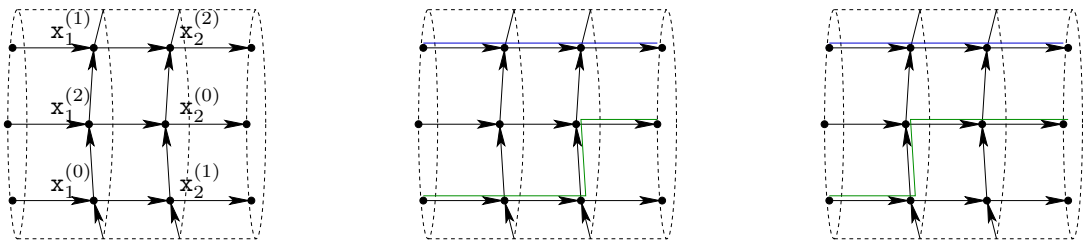}
    \end{center}
    \caption{The grid network $N(3,2)$ on a cylinder.}
    \label{fig:rig1}
\end{figure}
The $i$-th of the $m$ closed loops carries all the $\x_i^{(j)}$, $j=1, \ldots, n$. The color $j$ increases along the $n$ directed wires and decreases along the $m$ 
directed closed loops.

The left endpoints of the $n$ horizontal wires are called {\it {sources}} and denoted $0,1,\ldots,(n-1)$, and the right endpoints are called {\it {sinks}} and denoted $0',1',\ldots,(n-1)'$.  We fix the labeling conventions so that the first crossing to the right (resp. left) of the source (resp. sink) $j$ is labeled $\x_1^{(j)}$ (resp. $\x_m^{(j)}$).
We shall consider the \defn{highway paths} of \cite{LPsurf} that start at sources and end at sinks. There are three allowed ways for a highway path to pass through a vertex, shown in Figure \ref{fig:rig2}.  
\begin{figure}[h!]
    \begin{center}
    \input{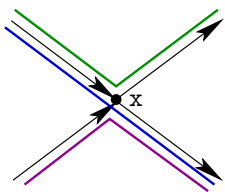}
    \end{center}
    \caption{The three allowed ways for a highway path to pass through a vertex.}
    \label{fig:rig2}
\end{figure}
At each such crossing a path picks up {\it the weight $1$ if it turned, or the weight $x$ of this vertex if it went through}.  The {\it weight} of a path $p$ is the product of weights it picks up at each of the vertices it passes through. The weight of a family of paths is the product of the weights of each path in the family. 

\begin{example}
The weight of the middle family of paths in Figure \ref{fig:rig1} is ${\textcolor{green} {\x_1^{(0)}} \textcolor{blue} {\x_1^{(1)} \x_2^{(2)}}}$.  The weight of the right family of paths in Figure \ref{fig:rig1} is
${\textcolor{green} {\x_2^{(0)}} \textcolor{blue} {\x_1^{(1)} \x_2^{(2)}}}$.
\qed
\end{example}

We call two highway paths \defn{non-crossing} if they do not share common edges.  We remark that the definition of highway paths already forbids two such paths crossing each other at a common vertex.  We call a family of paths non-crossing if any two paths in it are non-crossing.

Let $D = \D_s(\lambda/\mu)$ be a cylindric skew shape belonging to the cylinder $n-s$.  Thus the conjugate partitions $\lambda' = (\lambda'_1,\lambda'_2,\ldots,\lambda'_{n-s})$ and $\mu' = (\mu'_1,\mu'_2,\ldots,\mu'_{n-s})$ have $n-s$ parts (some of these parts can be zero).  
Let $r \in \Z/n\Z$.  Define an ordered $(n-s)$-tuple $S$ of the source vertices $\{0,1,\ldots,n-1\}$ by 
$$
S = (\mu'_1+r,\mu'_2-1+r,\mu'_3-2+r,\ldots,\mu'_i-(i-1)+r,\ldots,\mu'_{n-s}-(n-s-1)+r)
$$
where all numbers are taken in $\Z/n\Z$. In view of the correspondence (\ref{eq:tableau_weight}), the colors for the weights corresponding to the first boxes of the columns of $\lambda /\mu$ are exactly $S$. The condition that $\D_s(\lambda/\mu)$ is a valid cylindric skew shape implies that $\mu'_1 < \mu'_{n-s}+s+1=\mu'_{n-s}-(n-s-1) + n$, so that the ordered subset $S$ does not ``wrap" completely around the cylinder containing $N(n,m)$.
Thus for each column (the first column, in the following diagram) of the tableau $\lambda/\mu$, we consider the following type of a correspondence between the weight (\ref{eq:tableau_weight}) and a path on the following part of the grid on a cylinder.
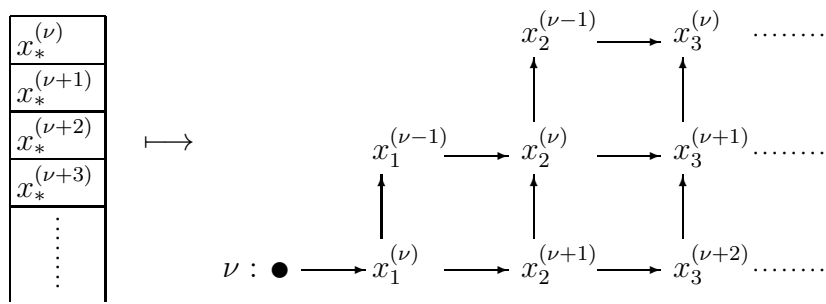
\begin{center}
\unitlength 12pt
\begin{picture}(26,10)
\multiput(-0.2,0.2)(3.0,0){2}{\line(0,1){9}}
\multiput(-0.2,9.2)(0,-1.5){5}{\line(1,0){3}}
\put(0,8){$x^{(\nu)}_*$}
\put(0,6.5){$x^{(\nu+1)}_*$}
\put(0,5.0){$x^{(\nu+2)}_*$}
\put(0,3.5){$x^{(\nu+3)}_*$}
\multiput(1.3,2.8)(0,-0.3){8}{\circle*{0.1}}
\put(4,5){$\longmapsto$}
\put(6.5,1){$\nu:$}
\put(8.3,1.2){\circle*{0.5}}
\put(9,1.2){\vector(1,0){2}}
\put(11.2,1){$x^{(\nu)}_1$}
\put(11.5,2.2){\vector(0,1){2}}
\put(11.2,4.7){$x^{(\nu-1)}_1$}
\multiput(13.5,1.2)(0,3.6){2}{\vector(1,0){2}}
\put(15.9,1){$x^{(\nu+1)}_2$}
\put(15.9,4.7){$x^{(\nu)}_2$}
\put(15.9,8.4){$x^{(\nu-1)}_2$}
\multiput(16.3,2.2)(0,3.7){2}{\vector(0,1){2}}
\multiput(18.3,1.2)(0,3.6){3}{\vector(1,0){2}}
\put(20.7,1){$x^{(\nu+2)}_3$}
\put(20.7,4.7){$x^{(\nu+1)}_3$}
\put(20.7,8.4){$x^{(\nu)}_3$}
\multiput(21,2.2)(0,3.7){2}{\vector(0,1){2}}
\multiput(23.3,1.2)(0.3,0){8}{\circle*{0.1}}
\multiput(23.3,4.9)(0.3,0){8}{\circle*{0.1}}
\multiput(23.3,8.6)(0.3,0){8}{\circle*{0.1}}
\end{picture}
\end{center}
Here we set $r=0$ and $\nu:=\mu'_1$.  Then the path starts from the bullet point (source) indicated by $\nu$.

Similarly, define an ordered $(n-s)$-tuple $R$ of the sink vertices $\{0',1',\ldots,(n-1)'\}$ by
$$
R = ((\lambda'_1-1+r)',(\lambda'_2-2+r)',(\lambda'_3-3+r)',\ldots,(\lambda'_i-i+r)',\ldots,(\lambda'_{n-s}-(n-s)+r)').
$$

For a highway path $p$ in $N(n,m)$ connecting a source $s$ to a sink $r$, let $\ell(p)$ be the degree of the weight of $p$.  Equivalently, $\ell(p)$ is the number of interior vertices that $p$ goes through (and does not turn).

\begin{proposition} \label{prop:noncross}
There is a weight-preserving bijection between families $(p_1,p_2,\ldots,p_{n-s})$ of non-crossing paths in $N(n,m)$, where $p_i$ connects $S_i$ to $R_i$ and has weight of degree $\ell(p_i) = \#\{\text{boxes in $i$-th column of $D$}\}$, and semistandard Young tableau of the cylindric skew shape $\D_s(\lambda/\mu)$ with the numbers $1,2,\ldots,m$.
\end{proposition}
\begin{proof}
Let $T: D \to \{1,2,\ldots,m\}$ be a cylindric semistandard Young tableau of shape $D$.  Let $C_i$ be the $i$-th column of $D$.  Then $T(C_i) \subset \{1,2,\ldots,m\}$.  We construct from $T(C_i)$ a path $p_i$ from $S_i$ to $R_i$, as follows (see also the diagram just before the proposition).  For $k = 1,2,\ldots,m$, at the $k$-th vertical closed loop $L_k$ in $N(n,m)$ we either (a) turn left at the first vertex of $L_k$ that we encounter, and turn right at the next vertex on $L_k$, or (b) go straight through $L_k$ at the first vertex of $L_k$ we encounter.  We choose option (a) whenever $k \notin T(C_i)$ and option (b) whenever $k \in T(C_i)$.  This completely determines $p_i$ since we know the source and the sink.  Note that we will pick up a weight of $\x_k^{(?)}$ if and only if option (b) is chosen, if and only if $k \in T(C_i)$.

Then we can verify that the map $T \mapsto (p_1,p_2,\ldots,p_{n-s})$ gives the desired bijection.
It is straightforward to check that the semistandard condition on $T$ corresponds to the non-crossing
condition for the paths $p_i$ and $p_{i+1}$ for any $1\leq i<n-s$.
For the remaining pair $p_{n-s}$ and $p_1$, recall that the shift
$(n-s,s)$ does not change the colors of the weights.
Then the cylindric semistandard condition ensures the non-crossing property of
the paths $p_{n-s}$ and $p_1$ as in the previous case.
\end{proof}

\begin{example}
 Consider the two families of paths in Figure \ref{fig:rig1}. They correspond to $n=3, m=2, s=1$, $\lambda = (2,1)$, $\mu = \emptyset$, $r=1$. In this case $S = (1,0)$ and 
$R=(2,0)$. 

The same families arise also for $n=3, m=2, s=1$, $\lambda = (2,2)$, $\mu = (1)$, $r=2$. In this case $S = (0,1)$ and 
$R=(0,2)$.
\qed
\end{example}

Recall that we have the symmetric group $S_m$ acting on $\mathbb \Q(\x_i^{(j)})$.

\begin{proposition} \label{prop:field}
The loop cylindric Schur functions $s_D^{(r)}(\x_1,\x_2,\ldots,\x_m)$ lie in the invariant subring $\Q[\x_i^{(j)}]^{S_m}$. 
\end{proposition}

Theorem \ref{thm:cylindric} follows immediately from Proposition \ref{prop:field} and Theorem \ref{thm:all}.

To prove Proposition \ref{prop:field}, we employ the network machinery of \cite{LPsurf}.  In particular, we will combine Lemmas \ref{lem:Smnetwork} and \ref{lem:measurements} below.

\begin{figure}[h!]
    \begin{center}
    \input{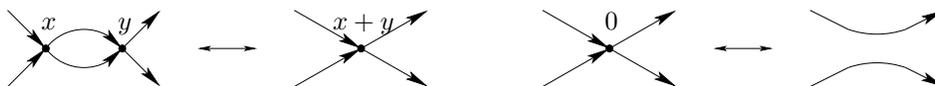}
    \end{center}
    \caption{Crossing merging/unmerging and crossing removal/creation moves.}
    \label{fig:rig3}
\end{figure}

\begin{figure}[h!]
    \begin{center}
    \input{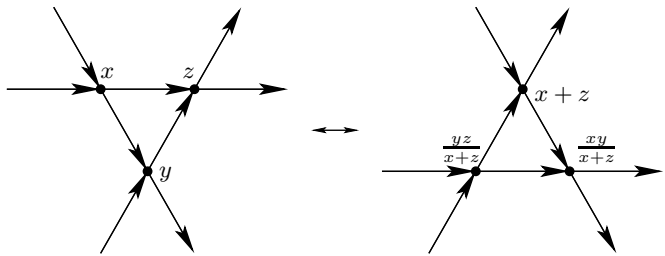}
    \end{center}
    \caption{The Yang-Baxter move.}
    \label{fig:rig4}
\end{figure}

\begin{lemma}[{\cite[Theorem 6.2]{LPsurf}}]
\label{lem:Smnetwork}
The $S_m$-action on the variables $\{\x_i^{(j)}\mid i =1,2,\ldots,m \text{ and } j \in \Z/n\Z\}$ can be realized as a sequence of the local moves from Figures \ref{fig:rig3} and \ref{fig:rig4}.
\end{lemma}

In other words, by applying a sequence of the moves in Figures \ref{fig:rig3} and \ref{fig:rig4}, we can change the grid network $N(n,m)$ into itself, so that the vertex weights have been changed by the action of any permutation $w \in S_m$ (acting via the geometric $R$-matrix).  

\begin{figure}[h!]
    \begin{center}
    \input{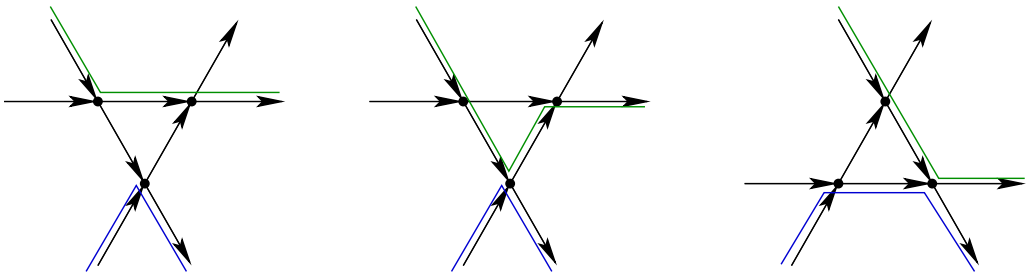}
    \end{center}
    \caption{Pairs of non-crossing highway paths with a fixed set of sources and sinks.}
    \label{fig:rig5}
\end{figure}

Consider any vertex weighted network $N$ on a surface as in \cite{LPsurf}, which in particular includes our grid networks $N(n,m)$ on a cylinder.  
For simplicity, we assume that any highway path $p$ in $N$ is guaranteed to never use a vertex more than once.

Let $S = (S_1,S_2,\ldots,S_a)$ be an ordered tuple of sources on $N$ (lying on the boundary of the surface) and $R = (R_1,R_2,\ldots,R_a)$ be an ordered tuple of sinks on $N$ of the same cardinality as $S$.  Let $M_{S,R,[p]}(N)$ denote the weight generating function of non-crossing path families $(p_1,p_2,\ldots,p_a)$, where $p_i$ goes from $S_i$ to $R_i$, and has a fixed homotopy class $[p_i]$.  Here $[p]=([p_1],[p_2],\ldots,[p_a])$ is an ordered tuple of homotopy classes of simple curves which begin and end on the boundary of the surface.

\begin{lemma}\label{lem:measurements}
Suppose $N$ and $N'$ are related by one of the local moves depicted in Figures \ref{fig:rig3} and \ref{fig:rig4}.  Then $M_{S,R,[p]}(N) = M_{S,R,[p]}(N')$.
\end{lemma}

\begin{proof}
This is a case-by-case verification for each of the local moves in Figures \ref{fig:rig3} and \ref{fig:rig4}.  The same calculation was carried out in \cite[Theorem 3.2]{LPsurf} in the case where $a = 1$; that is, it was checked that the claim is true when we consider a single path instead of a family of paths.

For the moves in Figure \ref{fig:rig3} the additional check for the case of a family of paths is trivial: at most two paths in the family can pass through the local picture, and when two non-crossing paths pass through the local picture they will both pick up no weight (the non-crossing condition forces the paths to turn at every vertex).  

For the Yang-Baxter move in Figure \ref{fig:rig4}, we need to consider the case of two paths passing through the local picture, and the case of three paths.  The latter case is again trivial because there the paths will pick up no weight.  For two paths we must do a case-by-case check.  We show one from the (short) list of needed calculations,
see Figure \ref{fig:rig5}. The equality to verify in this case is 
$$\textcolor{green} z \cdot \textcolor{blue} 1 + \textcolor{green} x \cdot \textcolor{blue} 1 = \textcolor{green} {(x+z)} \cdot \textcolor{blue} 1. $$
\end{proof}

\section{The conjecture}\label{sec:main}
\subsection{The tropicalization operation}
\label{sec:tropical}
Let $P(\x_1,\x_2,\ldots,\x_r)$ be a subtraction-free rational function with integer coefficients in the variables $\x_1,\ldots,\x_r$.  In other words, we assume that $P(\x_1,\x_2,\ldots,\x_r)$ can be written as a ratio of two polynomials with positive integer coefficients.  We define the \defn{tropicalization} $\trop(P)(x_1,x_2,\ldots,x_r)$ of $P(\x_1,\x_2,\ldots,\x_r)$ to be the piecewise-linear function in the variables $x_1,x_2,\ldots,x_r$, obtained from $P(\x_1,\x_2,\ldots,\x_r)$ by changing $+ \mapsto \min$, $\times \mapsto +$, and $\div \mapsto -$.  It is not difficult to see that $\trop(P)$ is well-defined: it does not depend on which subtraction-free expression of $P$ is used to compute it.  By convention, the polynomial $P = 1$ tropicalizes to the function $0$.

From now on, we fix positive integers $n$ and $m$.  Our rational functions will be in the variables $\{(\x_j^{(1)},\ldots,\x_j^{(n)}) \mid 1 \leq j \leq m\}$, and our piecewise-linear functions will be in the variables $\{(x_j^{(1)},\ldots,x_j^{(n)}) \mid 1 \leq j \leq m\}$.

The variables $\{(x_j^{(1)},\ldots,x_j^{(n)}) \mid 1 \leq j \leq m\}$ are ``coordinates" on the tensor products of row crystals in Section \ref{sec:R}, as follows.  Suppose that we are given the element $p=b_1\otimes\cdots\otimes b_m\in B^{1,s_1}\otimes\cdots\otimes B^{1,s_m}$ in a tensor product.  Then we let $x_j^{(i+j-1)}$ stand for the number of occurrences of the letter $i$ in $b_{m+1-j}$.  

The birational $R$-matrix of Section \ref{sec:birationalR} is a rational transformation all of whose coordinates are given by subtraction-free rational functions.  It therefore makes sense to ask for the tropicalization of this rational transformation, which will be a piecewise-linear transformation in the variables $\{(x_j^{(1)},\ldots,x_j^{(n)}) \mid 1 \leq j \leq m\}$.  The following well-known result connects the birational $R$-matrix with the combinatorial $R$-matrix defined in Section \ref{sec:R}.

\begin{theorem}[\cite{HHIKTT,Y}] \label{thm:Rtrop}
The tropicalization of the birational $R$-matrix $\R$ is exactly
the combinatorial $R$-matrix $R$ for tensor products of one row crystals.
\end{theorem}

The energy $H(b \otimes b')$ of Theorem \ref{thm:Hdefn} can also be described in a similar way.  In fact, this result follows quite easily from the description of energy given in Theorem \ref{thm:Hdefn}. 

\begin{proposition}
Let $y^{(i+1)}$ be the number of occurrences of the letter $i$ in $b$, and let $x^{(i)}$ be the number of occurrences of the letter $i$ in $b'$.  Then $H(b \otimes b') = \trop(\kappa_1)(x,y) $.
\end{proposition}

While we will not explicitly use these results, they do illustrate a general phenomenon: the combinatorial algorithms and quantities in the box-ball system can be lifted to the geometric/rational world.  Furthermore, these lifts are expressed in terms of loop symmetric functions.  Our work presents an extension of this philosophy to rigged configurations.

\subsection{Tropicalization of cylindric loop Schur functions}
\label{sec:formulation}
In this subsection we present our main conjecture: a loop symmetric function that tropicalizes to the shapes of the rigged configuration.  The positive integers $n$ and $m$ are always fixed.

Let $1 \leq s \leq (n-1)$.  
We define a partition $\lambda(s,r)$ recursively as follows.  Set $\lambda(s,0) = (n-s)^{m}$ to be a rectangle with $m$ rows of length $(n-s)$.  The partition $\lambda(s,r+1)$ is obtained from $\lambda(s,r)$ by removing the largest ribbon strip $R$ from $\lambda(s,r)$ that contains all the boxes in the bottom row of $\lambda(s,r)$, and such that $R$ has at most $n$ boxes.  Obviously we have $\lambda(s,0) \supset \lambda(s,1) \supset \cdots $, and the sequence of partitions eventually stabilize at the empty partition.

%

For example, suppose $n = 6, m = 7$, and $s =2$.  Then the shapes are
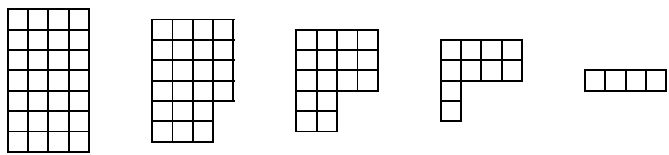
\[
\Yboxdim8pt\Yvcentermath1
\yng(4,4,4,4,4,4,4)\qquad
\yng(4,4,4,4,3,3)\qquad
\yng(4,4,4,2,2)\qquad
\yng(4,4,1,1)\qquad
\yng(4)
\]

It is possible for the ribbon strip $R$ to have less than $n$ boxes.  For example, suppose $n = 6, m =3$, and $s = 3$.  Then the shapes are
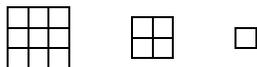
\[
\Yboxdim8pt\Yvcentermath1
\yng(3,3,3) \qquad \yng(2,2) \qquad \yng(1)
\]

Recall that $\D_s(\lambda(s,r))$ is obtained by propagating $\lambda(s,r)$ periodically in the plane by shifting $n-s$ steps to the right and $s$ steps up.

\begin{conjecture} \label{conj:shapes}
Let $p=b_1\otimes\cdots\otimes b_m\in B^{1,s_1}\otimes\cdots\otimes B^{1,s_m}$ and define $x_j^{(i+j-1)}$ to be the number of occurrences of the letter $i$ in $b_{m+1-j}$.  Then for $1 \leq s \leq n-1$ and an integer $r \geq 1$, we have
$$
\trop \left(\frac{s_{\D_s(\lambda(s,r-1))}^{(0)}}{s_{\D_s(\lambda(s,r))}^{(0)}} \right)(x_j^{(i)}) = \nu^{(s)}_r,
$$
where $\nu^{(s)}_r$ is the $r$-th part in the $s$-th shape of the rigged configuration $\Phi(p)$.
\end{conjecture}

Our conjecture completely determines the shapes of the rigged configuration $\Phi(p)$, but we have been unable to find a formula for the riggings themselves.

\begin{problem}
Find elements of $\Frac(\LSym_m)$ that tropicalize to the riggings $J^{(s)}$ of the rigged configuration $\Phi(p)$.
\end{problem}

\begin{example}
\label{ex:n=3m=3}
Let $n = 3$ and $m = 3$.  Then for $s = 1$, the shapes $\lambda(1,0),\lambda(1,1),\ldots$ are given by
\[
\Yboxdim8pt\Yvcentermath1
\yng(2,2,2)\qquad
\yng(2,1)\qquad
\emptyset\qquad
\emptyset\qquad
\cdots
\]
The cylindric loop Schur functions are (see Example \ref{ex:cylindric_Schur})
\begin{align*}
s_{\D_1(2,2,2)}^{(0)} =&\, \x_1^{(3)} \x_1^{(2)} \x_2^{(1)} \x_2^{(3)} \x_3^{(2)} \x_3^{(1)},\\
s_{\D_1(2,1)}^{(0)} =&\,
\x_1^{(3)} \x_1^{(2)} \x_2^{(1)} + \x_1^{(3)} \x_2^{(2)} \x_2^{(1)} + \x_1^{(3)} \x_2^{(2)} \x_3^{(1)} + 
\x_1^{(3)} \x_1^{(2)} \x_3^{(1)}\\
&+ \x_2^{(3)} \x_2^{(2)} \x_3^{(1)} + \x_1^{(3)} \x_3^{(2)} \x_3^{(1)} + \x_2^{(3)} \x_3^{(2)} \x_3^{(1)}.
\end{align*}
For $s=2$, $\lambda(2,0),\lambda(2,1),\ldots$ are given by
\[
\Yboxdim8pt\Yvcentermath1
\yng(1,1,1)\qquad
\emptyset\qquad
\emptyset\qquad
\cdots
\]
In this case, we have to consider the shift $(n-s,s)=(1,2)$.
Then the polynomial $s_{\D_2(1,1,1)}^{(0)}$ consists of only one term:
\begin{align*}
s_{\D_2(1,1,1)}^{(0)}=\x_1^{(3)}\x_2^{(1)}\x_3^{(2)}.
\end{align*}

Let us use the following coordinates for the tableau representation of the path
$p = b_1 \otimes b_2 \otimes b_3$
\begin{equation}\label{eq:pcoords}
p=\fbox{$1^{x_3^{(3)}}2^{x_3^{(1)}}3^{x_3^{(2)}}$}\otimes
\fbox{$1^{x_2^{(2)}}2^{x_2^{(3)}}3^{x_2^{(1)}}$}\otimes
\fbox{$1^{x_1^{(1)}}2^{x_1^{(2)}}3^{x_1^{(3)}}$}
=:\fbox{$1^a2^b3^c$}\otimes\fbox{$1^d2^e3^f$}\otimes\fbox{$1^g2^h3^i$}.
\end{equation}
In these coordinates, we have
\begin{align*}
\trop s_{\D_1(2,2,2)}^{(0)} &=i+h+f+e+c+b,\\
\trop s_{\D_1(2,1)}^{(0)} &=\min\{i+h+f,\,i+d+f,\,i+d+b,\,i+h+b,\\
&\hspace{14.5mm}
e+d+b,\,i+c+b,\,e+c+b\},\\
\trop s_{\D_2(1,1,1)}^{(0)} &=i+f+c.
\end{align*}
Then our formula reads
$\nu^{(1)}_1=\trop s_{\D_1(2,2,2)}^{(0)}-\trop s_{\D_1(2,1)}^{(0)}$,
$\nu^{(1)}_2=\trop s_{\D_1(2,1)}^{(0)}$ and $\nu^{(2)}_1=\trop s_{\D_2(1,1,1)}^{(0)}$.

For a numerical comparison, let us consider the following paths:
\[
\{a,b,c,d,e,f,g,h,i\}=\{0, 1, c, 2, 3, 1, 3, 3, 4\}
\]
where the notation is as in \eqref{eq:pcoords}.
Then $\Phi(p)$ takes the following form.
\[
\begin{array}{c|lll}
\hline\hline
c&(\nu^{(1)}_1,J^{(1)}_1)&(\nu^{(1)}_2,J^{(1)}_2)&(\nu^{(2)}_1,J^{(2)}_1)\\
\hline
0&(8,-5)&(4,-4)&(5,-1)\\
1&(8,-5)&(5,-4)&(6,-2)\\
2&(8,-5)&(6,-4)&(7,-3)\\
3&(9,-5)&(6,-3)&(8,-4)\\
4&(10,-5)&(6,-2)&(9,-5)\\
5&(11,-5)&(6,-1)&(10,-6)\\
6&(12,-5)&(6,-1)&(11,-7)\\
7&(13,-5)&(6,-1)&(12,-8)\\
\hline
\end{array}
\]

As another example, let us consider the paths
\[
\{a,b,c,d,e,f,g,h,i\}=\{2, 1, c, 3, 1, 1, 0, 1, 2\}.
\]
Then $\Phi(p)$ takes the following form.
\[
\begin{array}{c|lll}
\hline\hline
c&(\nu^{(1)}_1,J^{(1)}_1)&(\nu^{(1)}_2,J^{(1)}_2)&(\nu^{(2)}_1,J^{(2)}_1)\\
\hline
0&(4,1)&(2,-1)&(3,-1)\\
1&(4,1)&(3,-1)&(4,-2)\\
2&(4,-1)&(4,-1)&(5,-3)\\
3&(5,-1)&(4,-1)&(6,-4)\\
4&(6,-1)&(4,-1)&(7,-5)\\
5&(7,-1)&(4,-1)&(8,-6)\\
6&(8,-1)&(4,-1)&(9,-7)\\
7&(9,-1)&(4,-1)&(10,-8)\\
\hline
\end{array}
\]
In both cases, we see that the shapes of the configurations coincide with our expression in terms of the cylindric loop Schur functions. 
\qed
\end{example}

\begin{example}
Let us consider the case when $\nu^{(2)}$ has more than one row.
For that purpose, we take $n=4$ and $m=4$.
For $s=2$, the partitions $\lambda(2,0),\lambda(2,1),\ldots$ are given by
\[
\Yboxdim8pt\Yvcentermath1
\yng(2,2,2,2)\qquad
\yng(2,1,1)\qquad
\emptyset\qquad
\emptyset\qquad
\cdots
\]
In this case, in order to construct $\D_2(\lambda(s,r))$,
we have to consider the shift $(n-s,s)=(2,2)$.
Then $s^{(0)}_{\D_2(2,2,2,2)}$ consists of one term:
\[
s^{(0)}_{\D_2(2,2,2,2)}=\x^{(4)}_1\x^{(3)}_1\x^{(1)}_2\x^{(4)}_2\x^{(2)}_3\x^{(1)}_3\x^{(3)}_4\x^{(2)}_4.
\]
Next, we see that the following 14 tableaux contribute to $s^{(0)}_{\D_2(2,1,1)}$:
\begin{align*}
&
\young(11,2,3)\quad
\young(11,2,4)\quad
\young(11,3,4)\quad
\young(12,2,3)\quad
\young(12,2,4)\quad
\young(12,3,4)\quad
\young(13,2,3)\\
&
\young(13,2,4)\quad
\young(13,3,4)\quad
\young(14,2,4)\quad
\young(14,3,4)\quad
\young(22,3,4)\quad
\young(23,3,4)\quad
\young(24,3,4)
\end{align*}
however the following tableau does not contribute to $s^{(0)}_{\D_2(2,1,1)}$:
\[
\young(14,2,3)
\]
Therefore we have the following expression:
\begin{align*}
s^{(0)}_{\D_2(2,1,1)}&=
\x^{(4)}_1\x^{(3)}_1\x^{(1)}_2\x^{(2)}_3+
\x^{(4)}_1\x^{(3)}_1\x^{(1)}_2\x^{(2)}_4+
\x^{(4)}_1\x^{(3)}_1\x^{(1)}_3\x^{(2)}_4+
\x^{(4)}_1\x^{(3)}_2\x^{(1)}_2\x^{(2)}_3\\
&+
\x^{(4)}_1\x^{(3)}_2\x^{(1)}_2\x^{(2)}_4+
\x^{(4)}_1\x^{(3)}_2\x^{(1)}_3\x^{(2)}_4+
\x^{(4)}_1\x^{(3)}_3\x^{(1)}_2\x^{(2)}_3+
\x^{(4)}_1\x^{(3)}_3\x^{(1)}_2\x^{(2)}_4\\
&+
\x^{(4)}_1\x^{(3)}_3\x^{(1)}_3\x^{(2)}_4+
\x^{(4)}_1\x^{(3)}_4\x^{(1)}_2\x^{(2)}_4+
\x^{(4)}_1\x^{(3)}_4\x^{(1)}_3\x^{(2)}_4+
\x^{(4)}_2\x^{(3)}_2\x^{(1)}_3\x^{(2)}_4\\
&+
\x^{(4)}_2\x^{(3)}_3\x^{(1)}_3\x^{(2)}_4+
\x^{(4)}_2\x^{(3)}_4\x^{(1)}_3\x^{(2)}_4.
\end{align*}

For the path $p=b_1\otimes b_2\otimes b_3\otimes b_4$,
we introduce the following coordinates:
\begin{align*}
p&=
\fbox{$1^{x^{(4)}_4} 2^{x^{(1)}_4} 3^{x^{(2)}_4} 4^{x^{(3)}_4}$} \otimes
\fbox{$1^{x^{(3)}_3} 2^{x^{(4)}_3} 3^{x^{(1)}_3} 4^{x^{(2)}_3}$} \otimes
\fbox{$1^{x^{(2)}_2} 2^{x^{(3)}_2} 3^{x^{(4)}_2} 4^{x^{(1)}_2}$} \otimes
\fbox{$1^{x^{(1)}_1} 2^{x^{(2)}_1} 3^{x^{(3)}_1} 4^{x^{(4)}_1}$}\\
&=\fbox{$1^a 2^b 3^c 4^d $}\otimes \fbox{$1^e 2^f 3^g 4^h $}\otimes
\fbox{$1^i 2^j 3^k 4^l $}\otimes \fbox{$1^m 2^n 3^o 4^{p\mathstrut} $}.
\end{align*}
Then we have
\begin{align*}
\trop s^{(0)}_{\D_2(2,2,2,2)}=&\,p+o+l+k+h+g+d+c,\\
\trop s^{(0)}_{\D_2(2,1,1)}=&\,\min(
p + o + l + h,\, p + o + l + c,\, p + o + g + c,\, p + j + l + h,\\
&\hspace{10mm} p + j + l + c,\, p + j + g + c,\, p + e + l + h,\, p + e + l + c,\\
&\hspace{10mm} p + e + g + c,\, p + d + l + c,\, p + d + g + c,\, k + j + g + c,\\
&\hspace{10mm} k + e + g + c,\, k + d + g + c).
\end{align*}
Then our formula reads $\nu^{(2)}_1=\trop s^{(0)}_{\D_2(2,2,2,2)}-\trop s^{(0)}_{\D_2(2,1,1)}$
and $\nu^{(2)}_2=\trop s^{(0)}_{\D_2(2,1,1)}$.

For a numerical comparison, let us consider the following paths;
\[
\{a,b,c,d,e,f,g,h,i,j,k,l,m,n,o,p\}
=\{3, 2, c, 3, 3, 3, 1, 0, 0, 3, 0, 2, 1, 0, 3, 3\}.
\]
Then $\Phi(p)$ takes the following form.
\[
\begin{array}{c|llllll}
\hline\hline
c & (\nu^{(1)}_1, J^{(1)}_1) & (\nu^{(1)}_2, J^{(1)}_2) & (\nu^{(1)}_3, J^{(1)}_3) &
 (\nu^{(2)}_1, J^{(2)}_1) & (\nu^{(2)}_2, J^{(2)}_2) & (\nu^{(3)}_1, J^{(3)}_1)\\
\hline
0 & (8,-2) & (8,-2) & (4,-1) & (8,4) & (4,0) & (8,-7) \\
1 & (9,-2) & (8,-1) & (4,-1) & (8,2) & (5,-1) & (8,-6) \\
2 & (10,-2) & (8,0) & (4,-1) & (8,0) & (6,-2) & (8,-5) \\
3 & (11,-2) & (8,1) & (4,-1) & (8,-2) & (7,-3) & (8,-4) \\
4 & (12,-2) & (8,2) & (4,-1) & (8,-4) & (8,-4) & (8,-3) \\
5 & (13,-2) & (8,2) & (4,-1) & (9,-5) & (8,-4) & (8,-3) \\
6 & (14,-2) & (8,2) & (4,-1) & (10,-6) & (8,-4) & (8,-3) \\
7 & (15,-2) & (8,2) & (4,-1) & (11,-7) & (8,-4) & (8,-3) \\
\hline
\end{array}
\]
These results coincide with our formula.
The riggings behave in a more complicated way than the shapes of the configurations
(some of the successive differences are equal to 2).
\qed
\end{example}

\begin{remark}
\label{rem:BerensteinKirillov}
Kirillov and Berenstein \cite{BK} computed explicit formulae for the rigged configuration $\Phi(p)$ for certain classes of paths.
Our philosophy in this paper is to produce formulae that lift to birational $R$-matrix invariants.  As far as we are aware, this is not the case for the formulae in \cite{BK}. 
\end{remark}

\section{Formula for the first shape of a rigged configuration}
\label{sec:first-shape}
In this section, we prove Conjecture \ref{conj:shapes} for the first shape $\nu^{(1)}$ of the rigged configuration.
In Section \ref{sec:main-theorem}, we write down the main result following
the description of Conjecture \ref{conj:shapes} and discuss its relation with the box-ball systems.
In Section \ref{sec:simplify} we show that related functions have a simplified expression in our case.
The proofs begin in Section \ref{sec:energy-tropical}.

\subsection{The main theorem and its implications}
\label{sec:main-theorem}
Let $p=b_1\otimes\cdots\otimes b_m\in B^{1,s_1}\otimes\cdots\otimes B^{1,s_m}$ be as usual, and let $x_j^{(i+j-1)}$ be the number of occurrences of the letter $i$ in $b_{m+1-j}$.  Let $\nu^{(1)}$ be the first partition in the rigged configuration $\Phi(p)$.  

\begin{theorem} \label{thm:main}
The parts of the first shape $\nu^{(1)}$ of a rigged configuration is given by the formula
$$
\nu^{(1)}_r = \trop{s_{\D_1(\lambda(1,r-1))}^{(0)}}(x_1,x_2,\ldots,x_m)
-\trop{s_{\D_1(\lambda(1,r))}^{(0)}}(x_1,x_2,\ldots,x_m).
$$
\end{theorem}
The proof is based on the results of \cite{S} and \cite{LP}.

Before embarking on the proof, let us discuss some implications of our theorem.
In order to illustrate the relationship between lengths of rows of $\nu^{(1)}$
and lengths of solitary waves of box-ball systems, we consider the following example.

\begin{example}
The example at the beginning of Section \ref{sec:bbs} corresponds to the following rigged configurations
($\ell\geq 3$).
\begin{center}
\unitlength 12pt
\begin{picture}(25,3.5)
\put(-2,1.5){$\Phi\left((T^{1,\ell})^t(p)\right)=$}
\put(5.8,2.5){$(1^{30})$}
\put(4,0){
\put(5.7,2.5){$22$}
\put(5.7,1.35){$24$}
\put(5.7,0.2){$26$}
\put(7,0){\yng(3,2,1)}
\put(10.7,2.5){$3t$}
\put(9.55,1.35){$4+2t$}
\put(8.4,0.2){$9+t$}
}
\put(19,0){
\put(-1.5,2.5){$-1$}
\put(-1.5,1.35){$-1$}
\put(0,1.1){\yng(2,2)}
\put(2.5,2.5){$-1$}
\put(2.5,1.35){$-2$}
}
\put(25.7,0){
\put(-0.8,2.5){$0$}
\put(0,2.2){\yng(2)}
\put(2.5,2.5){$0$}
}
\end{picture}
\end{center}
Here $t=0,1,\ldots,7$ corresponds to the time presented on the left of the paths
and $(1^{30})$ stands for the single column Young diagram of height 30.

If we compare this with the example at the beginning of Section \ref{sec:bbs},
we obtain the following interpretation.
Each row of $\nu^{(1)}$ corresponds to a solitary wave of $p$ of the same length.
Then, roughly speaking, the corresponding rigging specifies the position of the solitary wave.
Thus, if the rigging behaves like $\sim vt$, then the corresponding solitary wave
propagates at velocity $v$ if there are no scatterings with other waves.
\qed
\end{example}

The general statement of this phenomenon is the following theorem
which is a consequence of the $R$-invariance property of $\Phi$
(Theorem \ref{th:combR}).\footnote{This result is valid for the general case
$p\in B^{r_1,s_1}\otimes B^{r_2,s_2}\otimes\cdots\otimes B^{r_m,s_m}$
without any modification.}

\begin{theorem}[\cite{KOSTY}]
\label{th:ist}
Suppose that we put a sufficiently long tensor product of $u^{r,1}$ on
the right of $p\in B^{1,s_1}\otimes B^{1,s_2}\otimes\cdots\otimes B^{1,s_m}$.
Suppose that we have
\begin{align*}
\Phi(p)=\left\{
(\nu^{(0)}),(\nu^{(1)},J^{(1)}),\ldots, (\nu^{(r)},J^{(r)}),\ldots,(\nu^{(n-1)},J^{(n-1)})
\right\}.
\end{align*}
Then we have
\begin{align*}
\Phi(T^{r,s}(p))=\left\{
(\nu^{(0)}),(\nu^{(1)},J^{(1)}),\ldots, (\nu^{(r)},\hat{J}^{(r)}),\ldots,(\nu^{(n-1)},J^{(n-1)})
\right\},
\end{align*}
i.e., the only difference appears at $\hat{J}^{(r)}$.
Here, for the original string $(\nu^{(r)}_i,J^{(r)}_i)$, the new string is
$(\nu^{(r)}_i,\hat{J}^{(r)}_i)=(\nu^{(r)}_i,J^{(r)}_i+\min(s,\nu^{(r)}_i))$.
\end{theorem}

Combining Theorem \ref{th:ist} and the results of \cite{S_2006}, we deduce the following interpretation of $\nu^{(1)}$
which is the subject of our Theorem \ref{thm:main}.
We keep the notations and assumptions of Theorem \ref{th:ist}.
For $b_1\otimes b_2\in B^{1,s_1}\otimes B^{1,s_2}$, define
$\bar{H}(b_1\otimes b_2)=\min(s_1,s_2)-H(b_1\otimes b_2)$.
Removing $\nu^{(0)}$ and $J^{(1)}$ from $\Phi(p)$, we regard the following rigged configuration
\begin{align*}
(\bar{\nu},\bar{J})=\left\{
(\nu^{(1)}),(\nu^{(2)},J^{(2)}), (\nu^{(3)},J^{(3)}),\ldots,(\nu^{(n-1)},J^{(n-1)})
\right\}
\end{align*}
as a $\hat\sl_{n-1}$ rigged configuration.
Let $\bar{\Phi}^{-1}$ be the map obtained by adding 1 to all numbers of all tableaux
obtained from $\Phi^{-1}$.
Then we define
\[
\bar{\Phi}^{-1}(\bar{\nu},\bar{J})=\bar{b}_1\otimes\cdots\otimes \bar{b}_M
\in B^{1,\nu^{(1)}_1}\otimes\cdots\otimes B^{1,\nu^{(1)}_M}\qquad
(M=\text{length of }\nu^{(1)}).
\]
Here we assume that $\nu^{(1)}_1\leq \nu^{(1)}_2\leq\cdots$.
In the following description, we will identify the row
$\bar{b}_k=
\begin{array}{|c|c|c|c|}
\hline
c_1&c_2&\cdots&c_l\\
\hline
\end{array}$ and the solitary wave
$\fbox{$c_l$}\otimes\cdots\otimes\fbox{$c_2$}\otimes\fbox{$c_1$}$.
Then, for a sufficiently large $N$, the path $(T^{1,\ell})^N(p)$ contains
the solitons $\bar{b}_1,\bar{b}_2,\ldots,\bar{b}_M$ from left to right.
Moreover, the length of the spacing between $\bar{b}_{k}$ and $\bar{b}_{k+1}$
is no less than $\bar{H}(\bar{b}_{k}\otimes\bar{b}_{k+1})$.

According to the above observations we see that Theorem \ref{thm:main} provides the lengths of the solitary waves contained in any $p$
which will appear after sufficiently many applications of time evolutions.
As a corollary, we can easily obtain the upper bound for the number of solitons contained in a path.

Let $\lceil x \rceil$ be the ceiling function which gives the smallest integer not smaller than $x$
and $\lfloor x \rfloor$ be the floor function which gives the largest integer not exceeding $x$.
Recall that the partition $\lambda(1,r)$ is obtained by removing length $ \leq n$ ribbon strips
$r$ times from the rectangle $(n-1)^m$ (see Section \ref{sec:formulation}).  It is not hard to see that in this  ($s = 1$) case all the ribbon strips removed do indeed have length $n$, except possibly the last ribbon strip.
Thus if the partition $\lambda(1,r)$ is nonempty then it has $\lfloor m-rn/(n-1) \rfloor$ parts equal to $(n-1)$, and a single part equal to $n-1-r'$ where $r \equiv r' \mod (n-1)$ and $1 \leq r' \leq n-1$.

\begin{corollary}
\label{cor:number_of_solitons}
The number of solitons contained in a length $m$ path of type $\hat{\sl}_n$ is at most
\[
\left\lceil
\frac{(n-1)m}{n}
\right\rceil.
\]
\end{corollary}
\begin{proof}
According to Theorem \ref{thm:main}, if we have
$s^{(0)}_{\D_1(\lambda(1,r))}=1$, then $\nu^{(1)}_{r+1}=0$.  Since $\lambda(1,r) = \emptyset$ for $r > (n-1)m/n$, we obtain the desired expression.
\end{proof}

Finally let us say a few words about the explicit piecewise linear formula
for $\Phi^{-1}$ obtained in \cite{KSY,S_2006}.
One of the key steps of the construction is to take the tropical limit
of the so-called tau functions for the KP hierarchy \cite{JM}
(see \cite[Section 5]{S:review} for a brief summary of the proof.).
In this step, each string $(\nu^{(a)}_i,J^{(a)}_i)$ is explicitly connected
with a soliton of the ordinary KP hierarchy.

\subsection{Simplified expressions}
\label{sec:simplify}
In the current situation, we have the following realization of
the cylindric loop Schur functions $s^{(0)}_{D(\lambda(1,r))}$.
As in \cite{LP}, we define the following generating function:
\begin{align*}
\tau_k^{(a)}(\x_1,\x_2,\ldots,\x_m) := \sum_{\Gamma = \gamma_1 \leq \dotsc \leq \gamma_k}
\x_{\gamma_1}^{(a)} \x_{\gamma_2}^{(a-1)} \dotsc \x_{\gamma_k}^{(a-k+1)},
\end{align*}
where the sum is taken over all multisets $\Gamma \subseteq \{1,2,\ldots, m\}$ of cardinality $k$
such that no number occurs more than $n-1$ times.
Let us denote the tropicalization of $\tau_k^{(a)}$ by $\Theta_k^{(a)}$;
\begin{align*}
\Theta_k^{(a)} := \trop(\tau_k^{(a)}) = \min_{\Gamma = \gamma_1 \leq \dotsc \leq \gamma_k}
\left\{ x_{\gamma_1}^{(a)} +x_{\gamma_2}^{(a-1)}+ \cdots +x_{\gamma_k}^{(a-k+1)} \right\}.
\end{align*}

\begin{lemma} \label{lem:tau}
We have 
$$
s_{\D_1(\lambda(1,r))}^{(0)} = \tau_{(n-1)m-rn}^{(0)}.
$$
\end{lemma}

\begin{proof}
The shape $\D_1(\lambda(1,r))$ is obtained by propagating $\lambda(1,r)$ by the shift $(n-1,1)$.  Since the vertical shift is by 1, there is a bijection between boxes of $\lambda(1,r)$ and boxes in any single row of $\D_1(\lambda(1,r))$.  For example, if $n = 4, m = 4$, and $r = 1$, we have $\lambda(1,1) = (3,3,2)$.  The cylindric shape $\D_1(\lambda(1,r))$ looks like
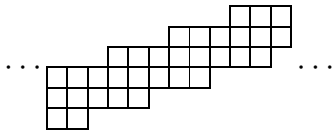
\[
\Yvcentermath1\Yboxdim8pt
\cdots
\young(:::::::::\X\X\X,::::::\X\X\X\X\X\X,:::\X\X\X\X\X\X\X\X,\X\X\X\X\X\X\X\X,\X\X\X\X\X,\X\X)
\cdots
\]
In this case $\lambda(1,r)$ has 8 boxes, and so does each row of $\D_1(\lambda(1,r))$.

Let $S$ be the cylindric semistandard tableau with shape $\D_1(\lambda(1,r))$.  By the above observation, the weight of $S$ is simply the weight of any single row $Z$ of $S$. Since the width of $\lambda(1,r)$ is $(n-1)$, the condition that $S$ is semistandard translates exactly to the condition that no number occurs in the weakly increasing sequence $Z$ more than $n-1$ times.  But the definition of $\tau_{(n-1)m-rn}^{(0)}$ is exactly the generating function of such rows.
\end{proof}

It follows from this and Theorem \ref{thm:cylindric} that $\tau_{(n-1)m-rn}^{(0)}$ lies in $\LSym_m$.

Now we see that Theorem \ref{thm:main} is equivalent to the following proposition.

\begin{proposition}
\label{prop:simple_formula}
Let $p\in B^{1,s_1}\otimes\cdots\otimes B^{1,s_m}$.
Then the $r$-th part of $\nu^{(1)}$ of $\Phi(p)$ is
\[
\nu^{(1)}_r=
\Theta^{(0)}_{(n-1)m-(r-1)n}(x_1,x_2,\ldots,x_m)-
\Theta^{(0)}_{(n-1)m-rn}(x_1,x_2,\ldots,x_m).
\]
\end{proposition}

\begin{example}
Let $n = 2$ and $m = 4$.  Then we have
\begin{align*}
\Theta_{4}^{(0)} &= x_1^{(2)} + x_2^{(1)} + x_3^{(2)} + x_4^{(1)},\\
\Theta_{2}^{(0)} &=
\min(x_1^{(2)}+x_2^{(1)}, x_1^{(2)}+x_3^{(1)}, x_1^{(2)}+x_4^{(1)},
x_2^{(2)}+x_3^{(1)}, x_2^{(2)}+x_4^{(1)}, x_3^{(2)}+x_4^{(1)}).
\end{align*}
Let the tableau representation of the path $p = b_1 \otimes b_2 \otimes b_3\otimes b_4$ be
\[
p=\fbox{$1^{x_4^{(2)}}2^{x_4^{(1)}}$}\otimes
\fbox{$1^{x_3^{(1)}}2^{x_3^{(2)}}$}\otimes
\fbox{$1^{x_2^{(2)}}2^{x_2^{(1)}}$}\otimes
\fbox{$1^{x_1^{(1)}}2^{x_1^{(2)}}$}
=:\fbox{$1^a2^b$}\otimes\fbox{$1^c2^d$}\otimes\fbox{$1^e2^f$}\otimes\fbox{$1^g2^h$}.
\]
Then we have
\begin{align*}
\Theta_{4}^{(0)}&=h+f+d+b,\\
\Theta_{2}^{(0)}&=\min(h+f,\,h+c,\,h+b,\,e+c,\,e+b,\,d+b).
\end{align*}
According to Proposition \ref{prop:simple_formula}, we have
$\nu^{(1)}_1=\Theta_{4}^{(0)}-\Theta_{2}^{(0)}$ and $\nu^{(1)}_2=\Theta_{2}^{(0)}$.
Furthermore, it also follows that $\nu^{(1)}_3 = 0$.
Note that these are also the lengths of the two solitons in the corresponding box-ball system.
\qed
\end{example}

\begin{example}
Let $n = 3$ and $m = 3$.  Then we have
\begin{align*}
\Theta_{6}^{(0)} &= x^{(3)}_1+x^{(2)}_1+x^{(1)}_2+x^{(3)}_2+x^{(2)}_3+x^{(1)}_3,\\
\Theta_{3}^{(0)} &= \min\left\{
x^{(3)}_1+x^{(2)}_1+x^{(1)}_2,\,x^{(3)}_1+x^{(2)}_1+x^{(1)}_3,\,x^{(3)}_1+x^{(2)}_2+x^{(1)}_2,\,x^{(3)}_1+x^{(2)}_2+x^{(1)}_3,
\right.\\
&\hspace{14.5mm}
\left.
x^{(3)}_1+x^{(2)}_3+x^{(1)}_3,\,x^{(3)}_2+x^{(2)}_2+x^{(1)}_3,\,x^{(3)}_2+x^{(2)}_3+x^{(1)}_3
\right\}.
\end{align*}
Let the tableau representation of the path $p = b_1 \otimes b_2 \otimes b_3$ be
\[
p=\fbox{$1^{x_3^{(3)}}2^{x_3^{(1)}}3^{x_3^{(2)}}$}\otimes
\fbox{$1^{x_2^{(2)}}2^{x_2^{(3)}}3^{x_2^{(1)}}$}\otimes
\fbox{$1^{x_1^{(1)}}2^{x_1^{(2)}}3^{x_1^{(3)}}$}
=:\fbox{$1^a2^b3^c$}\otimes\fbox{$1^d2^e3^f$}\otimes\fbox{$1^g2^h3^i$}.
\]
In this coordinate, we have
\begin{align*}
\Theta_{6}^{(0)} &=i+h+f+e+c+b,\\
\Theta_{3}^{(0)} &=\min\{i+h+f,\,i+h+b,\,i+d+f,\,i+d+b,\\
&\hspace{14.5mm}
i+c+b,\,e+d+b,\,e+c+b\}.
\end{align*}
Then by Proposition \ref{prop:simple_formula}, we have $\nu^{(1)}_1=\Theta_{6}^{(0)}-\Theta_{3}^{(0)}$
and $\nu^{(1)}_2=\Theta_{3}^{(0)}$.
Compare with Example \ref{ex:n=3m=3}.
\qed
\end{example}

\subsection{Energy as a tropical polynomial}
\label{sec:energy-tropical}
Recall that in Definition \ref{def:energy} we defined the quantity $E^{r,s}(p)=\sum_{k=1}^m H(u^{(k)}\otimes b_k)$ for $p = b_1 \otimes b_2 \otimes \cdots \otimes b_m$, where $u^{(1)} := u_{r,s}$.  We now define the piecewise-linear function $E_\ell(x_1,\ldots,x_m)$ to be equal to $E^{1,\ell}$, expressed in terms of the variables $x_j^{(i)}$, as usual.  In the following, we caution the reader that our $b_1 \otimes b_2 \otimes \cdots \otimes b_m$ corresponds to the tensor product $b_m \otimes \cdots \otimes b_2 \otimes b_1$ in \cite{LP}.

\begin{theorem}\label{thm:LP}
We have
\begin{align}\label{eq:LP_formula}
E_{\ell}(x_1, \ldots, x_m)= \min_{0 \leq i \leq (n-1)m}\left( \left\lceil\frac{i}{n}\right\rceil  \cdot \ell + \Theta^{(0)}_{(n-1)m-i}(x_1,\ldots,x_m)\right).
\end{align}
\end{theorem}
\begin{proof}
Lemma 3.1 in \cite{LP} gives a rational lifting of the energies $E_\ell(x_1,\ldots,x_m)$.  Taking into account the reversal of the tensor product, we have that $E_\ell(x_1,\ldots,x_m)$ is the tropicalization of 
$$
\sum_{i=0}^{(n-1)m} \tau^{(0)}_{(n-1)m-i}(\x_1,\x_2,\ldots,\x_{m}) \x_{m+1}^{(m+i)} \x_{m+1}^{(m+i-1)}\cdots 
\x_{m+1}^{(m+2)}\x_{m+1}^{(m+1)}
$$
%
%
where when we pass to tropicalization, we take 
$$(x_{m+1}^{(0)},x_{m+1}^{(1)},\ldots,x_{m+1}^{(n-1)}) = (0,0,\ldots,0,\ell,0,\ldots,0) \text{ with } x_{m+1}^{(m+1)} = \ell.$$
After tropicalization we get 
$x_{m+1}^{(m+1)} + x_{m+1}^{(m+2)}\cdots + x_{m+1}^{(m+i)} = \left\lceil\frac{i}{n}\right\rceil \cdot \ell$. 
\end{proof}

The formula in Theorem \ref{thm:LP} can be made more efficient by removing some of the terms in the minimum.  Namely, we note that the minimum is achieved when $i$ is divisible by $n$.

\begin{corollary} \label{cor:E}
We have
\begin{align}\label{eq:rational_E}
E_\ell(x_1,\ldots,x_m) = \min_{0 \leq i \leq {\frac{(n-1)m}{n}}}\left( i \cdot \ell + \Theta^{(0)}_{(n-1)m-in}(x_1,\ldots,x_m)\right).
\end{align}
\end{corollary}

\begin{proof}
To begin with we claim that if $a>a'$, we have $\Theta^{(0)}_{(n-1)m-a'}\geq \Theta^{(0)}_{(n-1)m-a}$.
Let $k'=(n-1)m-a'$ and $k=(n-1)m-a$.
Let $\x_{\gamma_1}^{(0)} \x_{\gamma_2}^{(-1)} \dotsc \x_{\gamma_{k'}}^{(-k'+1)}$
be a term of $\tau^{(0)}_{k'}$.
By dropping the last $(a-a')$ factors we obtain
$\x_{\gamma_{1}}^{(0)} \x_{\gamma_{2}}^{(-1)} \dotsc \x_{\gamma_{k}}^{(-k+1)}$,
which is easily seen to be one of the monomials occurring in $\tau^{(0)}_{k}$.
Since after tropicalization the variables satisfy $x_j^{(i)}\geq 0$,
the removed terms contribute non-negatively. Thus the claim follows.

Thus in the minimum of (\ref{eq:LP_formula}), for each subset of terms involving the same multiple 
$\left\lceil\frac{i}{n}\right\rceil  \cdot \ell$, we only need to keep one of the terms.
Thus we obtain (\ref{eq:rational_E}).
\end{proof}

\subsection{A lemma for convex tropical polynomials}
For a sequence $$0 = A_0, A_1, \ldots, A_{n-1}, A_{n}$$ of nonnegative integers,
define the tropical polynomial $$A(x) := \min(A_n, x+A_{n-1}, ..., (n-1)x+A_1, nx).$$ 
\begin{lemma} \label{lem:convex}
Suppose $0 = A_0, A_1, \ldots, A_{n-1}, A_{n}$ satisfy the convexity condition $$A_{i-1} + A_{i+1} \geq 2 A_i.$$  
Then the differences $$\Delta_A:=(A(1)-A(0), A(2)-A(1), \ldots)$$ form a partition.
Furthermore, this partition is conjugate to the
partition $$(A_n-A_{n-1}, A_{n-1}-A_{n-2}, \ldots, A_2-A_1, A_1).$$
\end{lemma}

\begin{example}
Consider $$A(x)=\min(7, 2+x, 2x).$$
Then its values at $x=0,1,2,\ldots$ are $0,2,4,5,6,7,7,7, \ldots.$,
thus we have 
$$\Delta_A=(2,2,1,1,1,0,0,\ldots).$$
The conjugate partition is $(5,2) = (7-2, 2)$.
\qed
\end{example}

\begin{proof}
It is easy to see that the minimum in $A(x)$ is achieved by the term $nx$ if and only if $A_1 \geq x$.
This means that the first $A_1$ terms in the sequence 
$\Delta_A$ are equal to $n$, while all subsequent ones are smaller than $n$.
Then if we conjugate, we conclude the smallest part of the conjugate partition is $A_1$.
When $x>A_1$, we have
$$A(x) = A_1+ \min((A_n-A_1), x+(A_{n-1}-A_1),\ldots, (n-2)x+(A_2-A_1), (n-1)x).$$
By repeating the same argument we see that the letter $(n-1)$ appears
$A_2-2A_1$ times in the sequence $\Delta_A$.
Note that by the convexity condition, we have $A_0+A_2=A_2\geq 2A_1$.
Next, when $x>A_2-A_1$, we have
$$A(x) = A_2+ \min((A_n-A_2), x+(A_{n-1}-A_2),\ldots, (n-3)x+(A_3-A_2), (n-2)x)$$
and conclude that the letter $(n-2)$ appears $A_3-A_2-(A_2-A_1)=A_3-2A_2+A_1$
times in the sequence $\Delta_A$.
Again we have $A_3-2A_2+A_1\geq 0$ by the convexity relation.
We repeat the same procedure to show the claim.
\end{proof}

\subsection{Log-concavity via cell transfer}
We argue that the tropicalizations of $\tau^{(0)}_{(n-1)m-in}$ satisfy the convexity condition above.

\begin{lemma} \label{lem:convextau}
 We have 
$$
\Theta^{(0)}_{(n-1)m-(i+1)n} + \Theta^{(0)}_{(n-1)m-(i-1)n} \geq 2 \Theta^{(0)}_{(n-1)m-in}.
$$
\end{lemma}

\begin{proof}
Let $N=(n-1)m-in$.
Then we will show that the difference
$$
(\tau^{(0)}_{N})^2 - \tau^{(0)}_{N-n} \cdot \tau^{(0)}_{N+n}
$$
is monomial positive. After tropicalization this means that in $2 \Theta^{(0)}_{N}$
the minimum is taken over more terms than in 
$\Theta^{(0)}_{N-n} + \Theta^{(0)}_{N+n}$, and thus the statement follows.

We shall use a variation of the {\it {cell transfer}} argument of \cite{LP0}.
Consider two terms, one a monomial from $\tau^{(0)}_{N+n}$
and the other from $\tau^{(0)}_{N-n}$.
We will provide an injection of such pairs into the set of ordered pairs of monomials in $\tau^{(0)}_{N}$.

Interpret the two terms as two single-row semistandard tableaux $A$ and $B$ of lengths
$N+n$ and $N-n$ respectively and denote them as
\begin{align*}
A=(a_1,a_2,\ldots,a_{N+n})\quad\text{and}\quad B=(b_1,b_2,\ldots,b_{N-n}),
\end{align*}
where $a_i\leq a_{i+1}$ and $b_i\leq b_{i+1}$ for any $i$.
Each of them does not have any number occurring more than $(n-1)$ times, according to the definition of $\tau$.

For $n \leq k\leq N$, define a pair of length $N$ single-row semistandard tableaux $A'_k$ and $B'_k$ as
\begin{align*}
A'_k &:= (a_1,a_2,\ldots,a_k,b_{k-n+1},b_{k-n+2},\ldots,b_{N-n}) \quad \text{and}\\
B'_k &:= (b_1,b_2,\ldots,b_{k-n},a_{k+1},a_{k+2},\ldots,a_{N+n}).
\end{align*}
Define a map $(A,B) \mapsto (A',B') = (A'_{k_0},B'_{k_0})$,
where $k_0$ is chosen to be maximal in $\{0,1,2,\ldots,N\}$,
such that $A'_{k_0}$ and $B'_{k_0}$ are single-row semistandard tableaux of shape $N$
such that no number appears more than $(n-1)$ times.
We claim that the parameter $k_0$ always exists, and that $(A,B) \mapsto (A',B')$ is injective.

We can determine the value of $k_0$ recursively.
First, set $K:= N$ and consider
\begin{align*}
&A'_{K}=(a_1,a_2,\ldots,a_{N}),\\
&B'_{K}=(b_1,b_2,\ldots,b_{N-n},a_{N+1},a_{N+2},\ldots,a_{N+n}).
\end{align*}
If this pair of tableaux satisfies the required conditions, then we have $k_0=N$.  Otherwise, one of the following two possibilities holds.
\begin{enumerate}
\item[(1)]
If $b_{K-n}>a_{K+1}$, we find the smallest $k > 0$ such that we have $b_{K-n-k} \leq a_{N-k+1}$.  Set $K':=K-k$.  If such a $k$ does not exist, we set $K':= 0$.
\item[(2)]
The value $b_{K-n}$ appears too many times in $B'_K$.
This can only happen if $b_{K-n}=a_{K+1}$.  Suppose $a_{K-k}<a_{K-k+1} = a_{K-k+2} = \cdots = a_{K+1}$.  Then we set $K'=K - k-1$.  (Note that $k$ always exists in this case, and $K'\geq 0$.)  
\end{enumerate}
If $A'_{K'}$ and $B'_{K'}$ satisfy the required conditions, then $k_0 = K'$ is the required value.  This is {\it always} the case if $K' = 0$.  Otherwise, we redefine $K:=K'$ and apply either (1) or (2) above.

The inverse map can be constructed in a similar way. Indeed, start with 
\begin{align*}
A=(a_1,a_2,\ldots,a_{N})\quad\text{and}\quad B=(b_1,b_2,\ldots,b_{N}),
\end{align*}
and define $A'_k$ and $B'_k$ as follows:
\begin{align*}
A'_k &:= (a_1,a_2,\ldots,a_{k},b_{k-n+1},b_{k-n+2},\ldots,b_{N}) \quad \text{and}\\
B'_k &:= (b_1,b_2,\ldots,b_{k-n},a_{k+1},a_{k+2},\ldots,a_{N}).
\end{align*}
Define a map $(A,B) \mapsto (A',B') = (A'_{k_0},B'_{k_0})$,
where $k_0$ is chosen to be maximal in $\{n,\ldots,N\}$,
such that $A'_{k_0}$ and $B'_{k_0}$ are single-row semistandard tableaux of shape $N$
such that no number appears more than $(n-1)$ times.
It is not true anymore that the parameter $k_0$ always exists. When it does exist however, this map and the one previously defined are inverses of each other.  
The proof is similar to the argument above, stemming from the observation that the value of $k$ that does not work in one case also cannot work in the other.

The existence of the inverse map proves injectivity. 
It follows that $(\tau^{(0)}_{N})^2 - \tau^{(0)}_{N-n} \cdot \tau^{(0)}_{N+n}$
is monomial positive.
\end{proof}

\subsection{Proof of Theorem \ref{thm:main}}

Apply Lemma \ref{lem:convex} to 
$$
A(x) = \min_{0 \leq i \leq {\frac{(n-1)m}{n}}} (i \cdot x +  \Theta^{(0)}_{(n-1)m-in}).
$$
We can do this because of Lemma \ref{lem:convextau}.  By Corollary \ref{cor:E}, 
we conclude that the partition formed by
$$\Theta^{(0)}_{(n-1)m-in} - \Theta^{(0)}_{(n-1)m-(i+1)n}$$
is conjugate to partition $(E_1, E_2-E_1, \ldots)$.

Now apply Theorem \ref{th:E=Q} to the quantities $E_\ell(x_1,\ldots,x_m) := E^{1,\ell}(p)$.  We deduce that the partition conjugate to $(E_1, E_2-E_1, \ldots)$ is equal to the partition $\nu^{(1)}$ of the rigged configuration.
Since Proposition \ref{prop:simple_formula} is equivalent to Theorem \ref{thm:main}, the result follows.
\qed

\end{document}